\documentclass[11pt]{article}
 
\usepackage[margin=0.8in]{geometry} 
\usepackage{amsmath,amsthm,amssymb}
\usepackage{mathrsfs}
\usepackage[T1]{fontenc}
\usepackage[english]{babel}
\usepackage{graphicx}
\usepackage{color} 

\usepackage{hyperref}
\usepackage{afterpage}

\usepackage{color}

\usepackage{mathtools}
\usepackage{atbegshi}
\AtBeginDocument{\AtBeginShipoutNext{\AtBeginShipoutDiscard}}
\usepackage{pdfpages} 
\usepackage{longtable} 
\usepackage{amsmath,amsthm,amsfonts,amssymb,relsize,bm} 
\usepackage{xparse,xcolor} 
\usepackage{graphicx,float} 
\usepackage{etoolbox}
\usepackage[shortlabels]{enumitem} 
\usepackage{tikz-cd} 
\usepackage{cite} 
\usepackage[normalem]{ulem}
\usepackage{comment}
\usepackage{tocloft}

\numberwithin{equation}{section}

\theoremstyle{plain}
\newtheorem{theorem}{Theorem}[section]

\newtheorem*{theorem*}{Theorem}
\newtheorem{lemma}{Lemma}[section]
\newtheorem{corollary}{Corollary}[section]

\theoremstyle{definition}

\newtheorem*{definition*}{Definition}

\newtheorem{remark}{Remark}[section]

\begin{document}
 
 \title{Poisson summation formula and Index Transforms}

\author{{Pedro Ribeiro, Semyon Yakubovich}} 
\thanks{


Department of Mathematics, Faculty of Sciences of University of Porto, Rua do Campo Alegre,  687; 4169-007 Porto (Portugal). 

\,\,\,\,\,E-mail of the corresponding author: pedromanelribeiro1812@gmail.com}
\date{}
 
\maketitle

\begin{abstract} Following the work of the second author, a class of summation formulas
attached to index transforms is studied in this paper. Our primary
results concern summation and integral formulas with respect to the
second index of the Whittaker function $W_{\mu,\nu}(z)$.
\end{abstract}

\pagenumbering{arabic}

\section{Introduction}

In 1964, Wimp \cite{Wimp} discovered an inversion formula for an integral transformation involving the 
Whittaker function $W_{\mu,i\tau}(x)$, 
\[
F(\tau)=\intop_{0}^{\infty}W_{\mu,i\tau}(x)\,f(x)\,\frac{dx}{x^{2}},\,\,\,\,\,\mu\in\mathbb{R},\,\,\tau>0.
\]
Due to a well-known reduction formula for $W_{\mu,\nu}(z)$ (see
(\ref{Whittaker to MAcdonald}) below), when $\mu=0$, this integral
transform reduces to the direct Kontorovich-Lebedev transform \cite{Yakubovich_Index_Transforms}.

The first author of this paper \cite{ribeiro_product_bessel} has recently studied
summation formulas attached to products of Bessel functions and to
Dirichlet series in the Hecke class. A curious application of the
Kontorovich-Lebedev transform (taken with respect to the index of
$K_{i\tau}(x)$) was given in this study and its application was crucial in some
steps to deduce new summation formulas.

The second author \cite{Yakubovich_Lebedev} studied a class of summation
formulas associated with the index of the Macdonald function and proved 
that the formula
\begin{equation}
K_{0}(x)+2\,\sum_{n=1}^{\infty}K_{in\alpha}(x)=\frac{\pi}{\alpha}e^{-x}+\frac{2\pi}{\alpha}\sum_{n=1}^{\infty}e^{-x\cosh\left(\frac{2\pi n}{\alpha}\right)}\label{lebedev sum}
\end{equation}
holds for any positive real numbers $x$ and $\alpha$.

Since $W_{\mu,i\tau}(x)$ generalizes the kernel $K_{i\tau}(x)$,
one may wonder which kind of transformation formulas can arise when the Macdonald
function is replaced by the Whittaker function in (\ref{lebedev sum}).
It is, therefore, the purpose of this paper to study summation formulas
of the type
\begin{equation}
\sum_{n=1}^{\infty}W_{\mu,i\alpha n}(x),\,\,\,\,\,\,\alpha,x>0,\,\,\,\mu\in\mathbb{R},\label{wht series}
\end{equation}
where the summation is over the imaginary second index of $W_{\mu,i\tau}(x)$.
Despite the formal similarity with the Kapteyn and Neumann series
(cf. \cite{Baricz_Neumann, baricz_book, watson_bessel, masirevic_pogany_2019, masirevic_pogany_2020, masirevic_pogany}), the nature and structure of
the summation formulas attached to the imaginary indices of the $W_{\mu,i\tau}(x)$ and $K_{i\tau}(x)$
drastically differ from these classical settings.

We evaluate the series (\ref{wht series}) by using a weak variant
of the Poisson summation formula. This approach has the advantage
to give new integral representations involving the index of a class
of special functions. In fact, the class of integrals derived in this paper cannot
be found in the standard references containing explicit index transforms with respect to Whittaker's function, such as \cite{ryzhik} and the third volume of \cite{MARICHEV}, nor in some classical papers devoted to
the evaluation of such integrals \cite{becker_I, becker_II}.

This paper is organized in the following simple manner: in the next
section, we briefly present the basic special functions that make
their appearance in this paper, and in the third (i.e., final) section, we compute the
Fourier transform of $f(\tau):=W_{\mu,i\tau}(x)$, for fixed parameters
$\mu\in\mathbb{R}$ and $x>0$. Using this transform and appealing
to a very simple version of Poisson's summation formula using residue
theory \cite{stein, stein_fourier}, we find a summation formula for the series (\ref{wht series}).
Several new results are deduced from the resulting transformation
formula. For example, the identity (valid for $x,\alpha>0$ and $\mu\in \mathbb{R}$), 
\begin{align*}
W_{-\frac{1}{4},0}(x)\,W_{\frac{1}{4},0}(x)+2\sum_{n=1}^{\infty}W_{-\frac{1}{4},i\alpha n}(x)\,W_{\frac{1}{4},i\alpha n}(x)\\
=\frac{x}{\sqrt{2}\alpha}\left\{ e^{-\frac{x}{2}}K_{0}\left(\frac{x}{2}\right)+2\sum_{n=1}^{\infty}e^{-\frac{x}{2}\cosh\left(\frac{\pi n}{\alpha}\right)}\,K_{0}\left(\frac{x}{2}\cosh\left(\frac{\pi n}{\alpha}\right)\right)\right\} 
\end{align*}
will be obtained as Corollary \ref{corollary 8} below. Additionally, in Corollary \ref{corollary 9} we will be able to compute the Fourier transform
\[
\intop_{-\infty}^{\infty}\left|\Gamma\left(\frac{1}{2}-\mu+i\tau\right)\right|^{2}\,W_{\mu,i\tau}(x)\,e^{-i\tau\xi}\,d\tau=\pi\,2^{\mu+\frac{1}{2}}\sqrt{x}\,\Gamma(1-2\mu)\,e^{\frac{x\sinh^{2}(\xi/2)}{2}}\,D_{2\mu-1}\left(\sqrt{2x}\,\cosh\left(\frac{\xi}{2}\right)\right),
\]
which cannot be found in any of the references containing this kind
of integrals.

\section{Preliminaries}

The notion of the Whittaker function arises in the study of the Kummer
confluent hypergeometric function, $_{1}F_{1}(a;c;z)$, which is usually
defined by the power series [\cite{NIST}, p. 322, 13.2.2],
\[
_{1}F_{1}\left(a;c;z\right):=\sum_{k=0}^{\infty}\frac{(a)_{k}}{(c)_{k}}\,\frac{z^{k}}{k!},
\]
where $(a)_{k}:=\Gamma(a+k)/\Gamma(a)$ denotes the Pochhammer symbol.
A special combination of Kummer functions leads to one of the most
important confluent hypergeometric functions known as Tricomi's function
$\Psi(a,c;z)$, which is defined by
\begin{equation}
\Psi(a,c;z):=\frac{\Gamma(1-c)}{\Gamma(a-c+1)}\,_{1}F_{1}(a;c;z)+z^{1-c}\,\frac{\Gamma(c-1)}{\Gamma(a)}\,_{1}F_{1}(a-c+1;2-c;z).\label{Whittaker definition confluent}
\end{equation}

The Whittaker function $W_{\mu,\nu}(x)$ is the solution of Whittaker's
differential equation {[}\cite{whittaker_watson}, p. 337{]},
\begin{equation}
\frac{d^{2}u}{dx^{2}}+\left(\frac{\mu}{x}+\frac{1}{4x^{2}}-\frac{\nu^{2}}{4x^{2}}-\frac{1}{4}\right)u=0,\label{differential equation Whittaker}
\end{equation}
which is determined uniquely by the property (here, we suppose that
$\mu\in\mathbb{R}$)
\begin{equation}
W_{\mu,\nu}(x)\sim x^{\mu}e^{-\frac{x}{2}},\,\,\,\,x\rightarrow\infty.\label{characterization whittaker function}
\end{equation}
The Whittaker function can be also defined in terms of the Tricomi
function, $\Psi(a,c;z)$, in the following manner
\begin{equation}
W_{\mu,\nu}(z)=e^{-z/2}z^{c/2}\Psi\left(a,c;z\right),\label{Whittaker tricomi relation}
\end{equation}
where $a=\frac{1}{2}+\nu-\mu$ and $c=2\nu+1$. Invoking the asymptotic
expansions for the Kummer and Tricomi functions,
we can check that both definitions of the function $W_{\mu,\nu}(x)$
are equivalent, because the right-hand side of (\ref{Whittaker tricomi relation})
satisfies the differential equation (\ref{differential equation Whittaker}),
but also
\[
W_{\mu,\nu}(x)=O\left(x^{\text{Re}(\nu)+\frac{1}{2}}\right)+O\left(x^{\frac{1}{2}-\text{Re}(\nu)}\right),\,\,\,\,\nu\neq0\,\,\,\,x\rightarrow0^{+},
\]
\[
W_{\mu,0}(x)=O\left(x^{\frac{1}{2}}\log(x)\right),\,\,\,\,x\rightarrow0^{+}
\]
and (supposing $\mu\in\mathbb{R}$)
\[
W_{\mu,\nu}(x)=O\left(e^{-x/2}x^{\mu}\right),\,\,\,\,x\rightarrow\infty.
\]

There are several integral representation of the Whittaker function
that will be very useful in this paper. Although the natural setting
in the theory of summation formulas (and in analytic number theory
in general) is concerned with Mellin-Barnes integrals, there are other
very important representations of the function $W_{\mu,\nu}(z)$ that can be considered
as classical. For example, one has the representation {[}\cite{NIST},
p. 337, 13.16.5{]}, 
\begin{equation}
W_{\mu,\nu}(z)=\frac{z^{\nu+\frac{1}{2}}2^{-2\nu}}{\Gamma\left(\frac{1}{2}+\nu-\mu\right)}\,\intop_{1}^{\infty}e^{-\frac{tz}{2}}(t-1)^{\nu-\mu-\frac{1}{2}}(t+1)^{\nu+\mu-\frac{1}{2}}\,dt,\,\,\,\,\text{Re}(\nu-\mu)+\frac{1}{2}>0,\,\,\,|\arg(z)|<\frac{\pi}{2}.\label{Mehler Laplace W}
\end{equation}
The Whittaker function satisfies several recurrence relations (see
{[}\cite{NIST}, section 13.15 (i), page 336{]} for a few of them).
A particular relation which will be important in our derivation of Corollary
\ref{corollary 4} is {[}\cite{NIST}, p. 336, 13.15.10{]}
\begin{equation}
2\nu W_{\mu,\nu}(z)=\sqrt{z}\left(W_{\mu+\frac{1}{2},\nu+\frac{1}{2}}(z)-W_{\mu+\frac{1}{2},\frac{1}{2}-\nu}(z)\right).\label{Recurrence Whittaker!}
\end{equation}
When we set its first index as zero, the function $W_{\mu,\nu}(x)$ reduces
to the modified Bessel function of the second kind (i.e., the Macdonald
function) and, in a different arrangement of the indices, it can be
reduced to an elementary function {[}\cite{NIST}, p. 338, 13.18.2,
13.18.9{]},
\begin{equation}
W_{0,\nu}\left(2x\right)=\sqrt{\frac{2x}{\pi}}\,K_{\nu}(x),\label{Whittaker to MAcdonald}
\end{equation}
\begin{equation}
W_{\nu+\frac{1}{2},\nu}(x)=x^{\nu+\frac{1}{2}}e^{-\frac{x}{2}}.\label{Whittaker to exponential}
\end{equation}

\bigskip{}

Let us now briefly describe the function that arises as a consequence
of the reduction formula (\ref{Whittaker to MAcdonald}). Generally,
it can be seen from (\ref{differential equation Whittaker}) that
the modified Bessel function of the second kind (or Macdonald function),
$K_{\nu}(x)$, satisfies the differential equation
\[
x^{2}\frac{d^{2}u}{dx^{2}}+x\,\frac{du}{dx}-\left(x^{2}+\nu^{2}\right)u=0,
\]
and has the asymptotic behavior
\begin{equation}
K_{\nu}(x)=\left(\frac{\pi}{2x}\right)^{\frac{1}{2}}e^{-x}\left\{ 1+O\left(\frac{1}{x}\right)\right\} ,\,\,\,\,x\rightarrow\infty,\label{asymptotic K}
\end{equation}
\begin{equation}
K_{\nu}(x)=2^{\nu-1}\Gamma(\nu)x^{-|\text{Re}(\nu)|}+\text{o}\left(x^{-|\text{Re}(\nu)|}\right),\,\,\,\,x\rightarrow0^{+},\,\,\nu\neq0,\label{asymptotic small different zero}
\end{equation}
\begin{equation}
K_{0}(x)=-\log(x)+O(1),\,\,\,\,x\rightarrow0^{+}.\label{asymptotic small equal to zero}
\end{equation}

\bigskip{}

The study of summation formulas attached to the function $K_{\nu}(x)$ depend
on many of its integral representations. For example, the only feature
distinguishing different proofs of Watson's formula [\cite{watson_reciprocal}, p. 299, eq. (4)]
is the integral representation used for the Macdonald function $K_{\nu}(x)$.\footnote{Watson himself confesses in his paper (see pages 299 and 300) that
the referee suggested a much easier proof of his formula based on
the Poisson summation formula and the Basset representation for $K_{\nu}(x)$.} One consequence of (\ref{Mehler Laplace W}) and
the reduction formula (\ref{Whittaker to MAcdonald}) is the integral
representation {[}\cite{NIST}, p. 252, 10.32.8{]},
\[
K_{\nu}(z)=\frac{\sqrt{\pi}z^{\nu}}{2^{\nu}\Gamma\left(\nu+\frac{1}{2}\right)}\,\intop_{1}^{\infty}e^{-zu}\left(u^{2}-1\right)^{\nu-\frac{1}{2}}du,\,\,\,\,\text{Re}(\nu)>-\frac{1}{2},\,\,|\arg(z)|<\frac{\pi}{2}.
\]
In 1871, Schl\"afli proved that this representation for the Modified
Bessel function is equivalent to the integral {[}\cite{NIST}, p. 252,
10.32.9{]} (see also {[}\cite{watson_bessel}, pp. 185-186, 6.31{]} for more
historical details),
\begin{equation}
K_{\nu}(z)=\intop_{0}^{\infty}e^{-z\cosh(u)}\cosh(\nu u)\,du,\,\,\,\,\nu\in \mathbb{C},\,\,\,\,|\arg(z)|<\frac{\pi}{2}.\label{Schlafli representation}
\end{equation}
In particular, if we take $\nu=i\tau$, $\tau\in\mathbb{R}$, the
integral representation (\ref{Schlafli representation}) gives
\begin{equation}
K_{i\tau}(x)=\frac{1}{2}\intop_{-\infty}^{\infty}e^{-x\cosh(u)}e^{i\tau u}du,\,\,\,\,x>0,\label{Poisson schlafli}
\end{equation}
which will be very important in the sequel. In fact, seeking the right
generalization of (\ref{Poisson schlafli}) in the Whittaker setting
is fundamental to obtain a transformation formula for the infinite
series (\ref{wht series}).

From the representation (\ref{Poisson schlafli}), the second author
obtained a uniform estimate of $K_{i\tau}(x)$ with respect to its
index $\tau\in\mathbb{R}$ and argument $x>0$, namely
\begin{equation}
|K_{i\tau}(x)|\leq e^{-\delta|\tau|}\,K_{0}\left(x\cos(\delta)\right),\,\,\,\,\delta\in[0,\pi/2).\label{uniform estimate Macdonald index and X}
\end{equation}
This estimate will
be invoked several times throughout this paper and we shall always
make explicit reference to it during the main text.

\bigskip{}

The Whittaker function has several other reductions that, for historical
reasons (and just like the Macdonald function $K_{\nu}(z)$), possess
independent interest. One of such reductions is the well-known parabolic
cylinder function, $D_{\mu}(z)$. In the classical approach of \cite{whittaker_watson},
it can be also defined as the solution of the differential equation
\[
\frac{d^{2}u}{dx^{2}}+\left(\mu+\frac{1}{2}-\frac{x^{2}}{4}\right)u=0,
\]
which can be then rewritten in an equation of the form (\ref{differential equation Whittaker}),
describing $D_{\mu}(z)$ in terms of the Whittaker function,
\begin{equation}
D_{\mu}(z)=2^{\frac{\mu}{2}+\frac{1}{4}}z^{-\frac{1}{2}}W_{\frac{\mu}{2}+\frac{1}{4},\frac{1}{4}}\left(\frac{z^{2}}{2}\right).\label{cylinder function def}
\end{equation}

It is known that $D_{\mu}(z)$ is an entire function of its parameter
$\mu$. This function makes its appearance in our treatment of the
infinite series (\ref{wht series}). Due to its connection with the
function $W_{\mu,\nu}(z)$, there are several useful representations
of the parabolic cylinder function that can be derived from integral
representations available for the Whittaker function. One such representation,
which will be very important in this paper, is {[}\cite{handbook_marichev},
p. 20, 2.2.1.6{]} (see also relation 2.3.15.3 on page 343 of vol. I of \cite{MARICHEV}),
\begin{equation}
\intop_{0}^{\infty}u^{\mu-1}e^{-xu^{2}-yu}du=\frac{\Gamma(\mu)}{(2x)^{\mu/2}}\,e^{\frac{y^{2}}{8x}}\,D_{-\mu}\left(\frac{y}{\sqrt{2x}}\right),\,\,\,\,\text{Re}(\mu),\,\,\text{Re}(x)>0,\,\,y\in\mathbb{C}.\label{integral cylinder useful final corollary}
\end{equation}
From the reduction formulas for the Whittaker function, one can easily
check {[}\cite{NIST}, p. 308, 12.7.1{]}
\begin{equation}
D_{0}(z)=e^{-\frac{z^{2}}{4}}\label{D0 formula!}
\end{equation}
and that {[}\cite{NIST}, p. 308, 12.7.5{]} (cf. {[}\cite{whittaker_watson},
p. 341{]} and {[}\cite{NIST}, p. 338, 13.18.7{]})
\begin{equation}
D_{-1}(z)=\sqrt{\frac{\pi}{2}}e^{\frac{z^{2}}{4}}\text{erfc}\left(\frac{z}{\sqrt{2}}\right),\label{D-1 formula!}
\end{equation}
where $\text{erfc}(\cdot)$ denotes the complementary error function,
usually defined by {[}\cite{NIST}, p. 160, 7.2(i){]}
\begin{equation}
\text{erfc}\left(x\right)=\frac{2}{\sqrt{\pi}}\intop_{x}^{\infty}e^{-t^{2}}dt.\label{Definition complementary error function!}
\end{equation}

It can be inferred from (\ref{characterization whittaker function})
that, when $x$ is taken large, $D_{\mu}(x)$ will satisfy the asymptotic
formula (cf. {[}\cite{NIST}, p. 309, 12.9 (i){]}, {[}\cite{ERDELIY_TRANSCENDENTAL}, Vol. II, p. 122, eq. 8.4(1){]}), 
\begin{equation}
D_{\mu}(x)\sim x^{\mu}e^{-\frac{x^{2}}{4}},\,\,\,\,x\rightarrow\infty.\label{Whittaker asymptotic formula}
\end{equation}
In particular, from the reduction formula (\ref{D-1 formula!}), we
can also write the following asymptotic formula for the complementary error function {[}\cite{NIST}, p.
164, 7.12.1{]}
\begin{equation}
\text{erfc}(x)\sim\frac{1}{\sqrt{\pi}}\,\frac{e^{-x^{2}}}{x},\,\,\,\,x\rightarrow\infty.\label{asymptotic erfc infinity}
\end{equation}

For a comprehensive approach to the theory of the confluent hypergeometric function,
which includes also includes a detailed study of the particular examples
$W_{\mu,\nu}(z)$ and $D_{\mu}(z)$, the reader is encouraged to check
Buchholz's book \cite{bucholz}.

To close this section, we still have to introduce a class of special
functions not directly related with the hypergeometric function $_{1}F_{1}(a;c;z)$
but with the function $_{1}F_{2}(a;b,c;z)$. These functions will
appear in Corollary \ref{corollary 6}, where we shall prove a summation formula
involving the second index of the Lommel function $S_{\mu,i\tau}(x)$.
The Lommel functions are usually defined by
\[
s_{\mu,\nu}(z)=\frac{z^{\mu+1}}{(\mu-\nu+1)(\mu+\nu+1)}\,_{1}F_{2}\left(1;\frac{\mu-\nu+3}{2},\frac{\mu+\nu+3}{2};-\frac{1}{4}z^{2}\right),
\]
and
\begin{align*}
S_{\mu,\nu}(z) & =s_{\mu,\nu}(z)+\frac{2^{\mu-1}\Gamma\left(\frac{\mu-\nu+1}{2}\right)\Gamma\left(\frac{\mu+\nu+1}{2}\right)}{\sin(\pi\nu)}\\
 & \,\,\,\,\quad\,\,\,\,\,\,\,\;\times\,\left\{ \cos\left(\frac{1}{2}(\mu-\nu)\pi\right)J_{-\nu}(z)-\cos\left(\frac{1}{2}(\mu+\nu)\pi\right)J_{\nu}(z)\right\} 
\end{align*}
for $\nu\notin\mathbb{Z}$ and
\begin{align*}
S_{\mu,\nu}(z) & =s_{\mu,\nu}(z)+2^{\mu-1}\Gamma\left(\frac{\mu-\nu+1}{2}\right)\Gamma\left(\frac{\mu+\nu+1}{2}\right)\\
 & \,\,\,\quad\,\,\,\,\,\,\,\,\,\,\times\left\{ \sin\left(\frac{1}{2}(\mu-\nu)\pi\right)J_{\nu}(z)-\cos\left(\frac{1}{2}(\mu-\nu)\pi\right)\,Y_{\nu}(z)\right\} 
\end{align*}
for $\nu\in\mathbb{Z}$, where $J_{\nu}(z)$ and $Y_{\nu}(z)$ respectively
represent the Bessel functions of first and second kind. Unlike the
Macdonald and Whittaker functions, the Lommel functions possess a
weak decay when $z:=x\in\mathbb{R}$ goes to infinity. This decay
is somewhat analogous to the decay of the Bessel functions $J_{\nu}(z)$
and $Y_{\nu}(z)$. In fact, for fixed $\mu$ and $\nu$, one has the
asymptotic expansion {[}\cite{NIST}, p. 295, 11.9.9{]},
\begin{equation}
S_{\mu,\nu}(z)\sim z^{\mu-1}\sum_{k=0}^{\infty}(-1)^{k}\prod_{m=1}^{k}\left(\left(2m-1-\mu\right)^{2}-\nu^{2}\right)\,z^{-2k},\,\,\,\,z\rightarrow\infty,\,\,|\arg(z)|\leq\pi-\delta.\label{Decay property once more Lommel!}
\end{equation}

Having presented all the main special functions comprising our summation
formulas, we are now ready to move to the next section and prove the
main results of this paper. 

\section{A class of Summation formulas attached to the index of Whittaker's function}

As remarked at the introduction, the primary goal in this paper is
to find a transformation formula for an infinite series of the form
\begin{equation}
\sum_{n=1}^{\infty}W_{\mu,i\alpha n}(x),\,\,\,\,\,\,\alpha,x>0,\,\,\,\mu\in\mathbb{R}.\label{whittaker start}
\end{equation}
Such transformation will necessarily possess at its core the classical
Poisson summation formula \cite{guinand_poisson, stein, stein_fourier, titchmarsh_fourier_integrals}. In order
to apply this classical result from Analysis and Number Theory, we
need to compute the Fourier transform
\begin{equation}
F_{\mu,\xi}(x):=\intop_{-\infty}^{\infty}W_{\mu,i\tau}(x)\,e^{-i\tau\xi}\,d\tau.\label{Fourier transform definition}
\end{equation}

First, let us note that the infinite series (\ref{whittaker start})
with respect to the second imaginary index converges for any $x\in(0,X]$,
$X>0$, by virtue of the asymptotic expansion (cf. {[}\cite{Yakubovich_Index_Transforms},
p. 25, eq. (1.172){]} and \cite{Yakubovich_Saigo})
\begin{equation}
W_{\mu,i\tau}(x)=\sqrt{2x}\,e^{-\frac{\pi\tau}{2}}\tau^{\mu-\frac{1}{2}}\cos\left(\tau\log\left(\frac{x}{4\tau}\right)-\frac{\pi}{2}\left(\mu-\frac{1}{2}\right)+\tau\right)\left\{ 1+O\left(\frac{1}{\tau}\right)\right\} ,\,\,\,\,\tau\rightarrow\infty.\label{expansion index whittaker!}
\end{equation}
In an analogous way, if $x$ is any fixed positive real number, then
the integral (\ref{Fourier transform definition}) exists as an absolutely
convergent integral.

In order to compute the transform (\ref{Fourier transform definition}),
we appeal to the following representation for $W_{\mu,i\tau}(x)$
{[}\cite{handbook_marichev}, p. 458, eq. (3.30.2.4){]},
\begin{equation}
e^{-x/2}W_{\mu,i\tau}(x)=\frac{1}{2\pi i}\,\intop_{\sigma-i\infty}^{\sigma+i\infty}\frac{\Gamma\left(\frac{2s-2i\tau+1}{2}\right)\,\Gamma\left(\frac{2s+2i\tau+1}{2}\right)}{\Gamma\left(s-\mu+1\right)}\,x^{-s}ds,\label{Mellin barnes Whittaker}
\end{equation}
valid for $x>0$, $\mu\in\mathbb{R}$ and $\sigma>0$.\footnote{In fact, this representation is even valid for $\sigma>-\frac{1}{2}$.
For simplicity, however, we work under the assumption $\sigma>0$.} Now, our Fourier transform reduces to the evaluation of the integral
\begin{equation}
F_{\mu,\xi}(x)=e^{x/2}\intop_{-\infty}^{\infty}\frac{e^{-i\tau\xi}}{2\pi i}\,\intop_{\sigma-i\infty}^{\sigma+i\infty}\frac{\Gamma\left(\frac{2s-2i\tau+1}{2}\right)\,\Gamma\left(\frac{2s+2i\tau+1}{2}\right)}{\Gamma\left(s-\mu+1\right)}\,x^{-s}ds\,d\tau,\label{Fourier transform at the beginning}
\end{equation}
and, in order to compute it, we will interchange the orders of integration
by appealing to Fubini's theorem. In order to check that this procedure
is valid, it suffices to show that
\begin{equation}
\intop_{0}^{\infty}\intop_{0}^{\infty}\left|\frac{\Gamma\left(\frac{1}{2}+\sigma+i(t-\tau)\right)\,\Gamma\left(\frac{1}{2}+\sigma+i(t+\tau)\right)}{\Gamma\left(\sigma-\mu+1+it\right)}\,\right|dt\,d\tau<\infty,\label{integral to justify via Stirling}
\end{equation}
since the resulting iterated absolute integral is even with respect
to the variables of integration $t$ and $\tau$. In the next argument, we shall invoke
Stirling's formula, 
\begin{equation}
|\Gamma(\sigma+it)|\asymp_{\sigma}\begin{cases}
|t|^{\sigma-\frac{1}{2}}e^{-\frac{\pi|t|}{2}} & |t|\geq1\\
1 & |t|\leq1
\end{cases}.\label{first version stirling}
\end{equation}

Assume first $0\le\tau\leq1$: then we split the inner integral (with
respect to $t$) in the following form
\begin{align*}
\intop_{0}^{\infty}\left|\frac{\Gamma\left(\frac{1}{2}+\sigma+i(t-\tau)\right)\,\Gamma\left(\frac{1}{2}+\sigma+i(t+\tau)\right)}{\Gamma\left(\sigma-\mu+1+it\right)}\,\right|dt\\
=\left\{ \intop_{0}^{\tau+1}+\intop_{\tau+1}^{\infty}\right\} \left|\frac{\Gamma\left(\frac{1}{2}+\sigma+i(t-\tau)\right)\,\Gamma\left(\frac{1}{2}+\sigma+i(t+\tau)\right)}{\Gamma\left(\sigma-\mu+1+it\right)}\,\right|dt.
\end{align*}
Since $0\leq\tau\leq1$ by hypothesis, the first integral will be
obviously bounded by continuity of the integrand in the range $0\leq t\leq\tau+1$.
Next, in the second range we have that $t-\tau\geq1$ and $t\geq1$, 
and so we may appeal to Stirling's formula (\ref{first version stirling})
to obtain
\[
\intop_{\tau+1}^{\infty}\left|\frac{\Gamma\left(\frac{1}{2}+\sigma+i(t-\tau)\right)\,\Gamma\left(\frac{1}{2}+\sigma+i(t+\tau)\right)}{\Gamma\left(\sigma-\mu+1+it\right)}\,\right|dt\ll_{\sigma}\intop_{\tau+1}^{\infty}t^{\sigma+\mu-\frac{1}{2}}e^{-\frac{\pi t}{2}}\,dt,
\]
which is obviously convergent. Therefore, 
\begin{align*}
\intop_{0}^{1}\intop_{0}^{\infty}\left|\frac{\Gamma\left(\frac{1}{2}+\sigma+i(t-\tau)\right)\,\Gamma\left(\frac{1}{2}+\sigma+i(t+\tau)\right)}{\Gamma\left(\sigma-\mu+1+it\right)}\,\right|dtd\tau & =\intop_{0}^{1}\intop_{0}^{\tau+1}\left|\frac{\Gamma\left(\frac{1}{2}+\sigma+i(t-\tau)\right)\,\Gamma\left(\frac{1}{2}+\sigma+i(t+\tau)\right)}{\Gamma\left(\sigma-\mu+1+it\right)}\,\right|dtd\tau\\
 & +\intop_{0}^{1}\intop_{\tau+1}^{\infty}t^{\sigma+\mu-\frac{1}{2}}e^{-\frac{\pi t}{2}}\,dt\,d\tau<\infty.
\end{align*}
Next, assume that $\tau\geq1$: we can split the inner integral in
(\ref{integral to justify via Stirling}) as 
\[
\left\{ \intop_{0}^{\tau-1}+\intop_{\tau-1}^{\tau+1}+\intop_{\tau+1}^{\infty}\right\} \left|\frac{\Gamma\left(\frac{1}{2}+\sigma+i(t-\tau)\right)\,\Gamma\left(\frac{1}{2}+\sigma+i(t+\tau)\right)}{\Gamma\left(\sigma-\mu+1+it\right)}\,\right|dt.
\]
For $0\leq t\leq\tau-1$ or $t\geq\tau+1$, we know that $|t-\tau|\geq1$,
so that
\[
\intop_{0}^{\tau-1}\left|\frac{\Gamma\left(\frac{1}{2}+\sigma+i(t-\tau)\right)\,\Gamma\left(\frac{1}{2}+\sigma+i(t+\tau)\right)}{\Gamma\left(\sigma-\mu+1+it\right)}\,\right|dt\leq\intop_{0}^{\tau-1}\frac{(\tau^{2}-t^{2})^{\sigma}\,e^{-\pi\tau}}{\left|\Gamma\left(\sigma-\mu+1+it\right)\right|}dt. 
\]
If $1\leq\tau\leq2$, then the integral is trivially estimated as
$\ll\tau^{2\sigma+1}e^{-\pi\tau}$, while if $\tau>2$, we have that
\begin{align}
\intop_{0}^{\tau-1}\frac{(\tau^{2}-t^{2})^{\sigma}e^{-\pi\tau}}{\left|\Gamma\left(\sigma-\mu+1+it\right)\right|}dt & \leq\tau^{2\sigma+1}e^{-\pi\tau}+\intop_{1}^{\tau-1}\frac{(\tau^{2}-t^{2})^{\sigma}e^{-\pi\tau}}{t^{\sigma-\mu+\frac{1}{2}}}e^{\frac{\pi}{2}t}\,dt\nonumber \\
 & \leq\tau^{2\sigma+1}e^{-\pi\tau}+\tau^{\eta}e^{-\frac{\pi}{2}\tau},\label{split inequality at the beginning}
\end{align}
where $\eta:=\max\left\{ 2\sigma+1,\sigma+\mu+\frac{1}{2}\right\} $. This proves the bound
\begin{align*}
\intop_{1}^{\infty}\intop_{0}^{\tau-1}\left|\frac{\Gamma\left(\frac{1}{2}+\sigma+i(t-\tau)\right)\,\Gamma\left(\frac{1}{2}+\sigma+i(t+\tau)\right)}{\Gamma\left(\sigma-\mu+1+it\right)}\,\right|dt\,d\tau\\
\leq\intop_{1}^{2}\tau^{2\sigma+1}e^{-\pi\tau}\,d\tau+\intop_{2}^{\infty}\left(\tau^{2\sigma+1}e^{-\pi\tau}+\tau^{\eta}e^{-\frac{\pi}{2}\tau}\right)d\tau<\infty.
\end{align*}
Next, in the region $\tau-1\leq t\leq\tau+1$
we have that $|t-\tau|\leq1$, so that Stirling's formula (\ref{first version stirling}) gives
\[
\intop_{\tau-1}^{\tau+1}\left|\frac{\Gamma\left(\frac{1}{2}+\sigma+i(t-\tau)\right)\,\Gamma\left(\frac{1}{2}+\sigma+i(t+\tau)\right)}{\Gamma\left(\sigma-\mu+1+it\right)}\,\right|dt\leq\intop_{\tau-1}^{\tau+1}\frac{(t+\tau)^{\sigma}e^{-\frac{\pi}{2}(t+\tau)}}{\left|\Gamma(\sigma-\mu+1+it)\right|}\,dt.
\]
Once more, if $1\leq\tau\leq2$, then the integral can be estimated
as
\begin{align*}
\intop_{\tau-1}^{\tau+1}\frac{(t+\tau)^{\sigma}e^{-\frac{\pi}{2}(t+\tau)}}{\left|\Gamma(\sigma-\mu+1+it)\right|}\,dt & \leq\intop_{\tau-1}^{1}\frac{(t+\tau)^{\sigma}e^{-\frac{\pi}{2}(t+\tau)}}{\left|\Gamma(\sigma-\mu+1+it)\right|}\,dt+\intop_{1}^{\tau+1}\frac{(t+\tau)^{\sigma}e^{-\frac{\pi}{2}(t+\tau)}}{\left|\Gamma(\sigma-\mu+1+it)\right|}\,dt\\
 & \ll\tau^{\sigma}e^{-\pi\tau}+\intop_{1}^{\tau+1}\frac{(t+\tau)^{\sigma}e^{-\frac{\pi}{2}\tau}}{t^{\sigma-\mu+\frac{1}{2}}}\,dt\leq\tau^{\sigma}e^{-\pi\tau}+\tau^{\eta-\sigma}e^{-\frac{\pi}{2}\tau},
\end{align*}
where, as before, $\eta:=\max\left\{ 2\sigma+1,\sigma+\mu+\frac{1}{2}\right\} $.
On the other hand, if $\tau\geq2$, we obtain
\[
\intop_{\tau-1}^{\tau+1}\frac{(t+\tau)^{\sigma}e^{-\frac{\pi}{2}(t+\tau)}}{\left|\Gamma(\sigma-\mu+1+it)\right|}\,dt\leq\intop_{\tau-1}^{\tau+1}\frac{(t+\tau)^{\sigma}e^{-\frac{\pi}{2}\tau}}{t^{\sigma-\mu+\frac{1}{2}}}\,dt\ll\tau^{\mu+\frac{1}{2}}e^{-\frac{\pi}{2}\tau},
\]
which results in 
\begin{align*}
\intop_{1}^{\infty}\intop_{\tau-1}^{\tau+1}\left|\frac{\Gamma\left(\frac{1}{2}+\sigma+i(t-\tau)\right)\,\Gamma\left(\frac{1}{2}+\sigma+i(t+\tau)\right)}{\Gamma\left(\sigma-\mu+1+it\right)}\,\right|dt\,d\tau\\
\leq\intop_{1}^{2}\left(\tau^{\sigma}e^{-\pi\tau}+\tau^{\eta-\sigma}e^{-\frac{\pi}{2}\tau}\right)d\tau+\intop_{2}^{\infty}\tau^{\mu+\frac{1}{2}}e^{-\frac{\pi}{2}\tau}d\tau<\infty.
\end{align*}
Finally, for $t\geq\tau+1$, we can
apply the first estimate in (\ref{first version stirling}) in all
of the Gamma factors and get the inequality
\begin{align*}
\intop_{\tau+1}^{\infty}\left|\frac{\Gamma\left(\frac{1}{2}+\sigma+i(t-\tau)\right)\,\Gamma\left(\frac{1}{2}+\sigma+i(t+\tau)\right)}{\Gamma\left(\sigma-\mu+1+it\right)}\,\right|dt & \leq\intop_{\tau+1}^{\infty}\frac{\left(t^{2}-\tau^{2}\right)^{\sigma}e^{-\frac{\pi}{2}t}}{t^{\sigma-\mu+\frac{1}{2}}}\leq\intop_{\tau+1}^{\infty}t^{\sigma+\mu-\frac{1}{2}}e^{-\frac{\pi}{2}t}dt\\
=e^{-\frac{\pi}{2}\tau}\intop_{1}^{\infty}(u+\tau)^{\sigma+\mu-\frac{1}{2}}e^{-\frac{\pi}{2}u}du & \ll\tau^{\sigma+\mu-\frac{1}{2}}e^{-\frac{\pi}{2}\tau}+e^{-\frac{\pi}{2}\tau}\intop_{\tau}^{\infty}u^{\sigma+\mu-\frac{1}{2}}e^{-\frac{\pi}{2}u}du\\
 & \ll\tau^{\sigma+\mu-\frac{1}{2}}e^{-\frac{\pi}{2}\tau}+e^{-\frac{\pi}{2}\tau},
\end{align*}
which proves that
\[
\intop_{1}^{\infty}\intop_{\tau+1}^{\infty}\left|\frac{\Gamma\left(\frac{1}{2}+\sigma+i(t-\tau)\right)\,\Gamma\left(\frac{1}{2}+\sigma+i(t+\tau)\right)}{\Gamma\left(\sigma-\mu+1+it\right)}\,\right|dt\,d\tau\ll\intop_{1}^{\infty}\left\{ \tau^{\sigma+\mu-\frac{1}{2}}e^{-\frac{\pi}{2}\tau}+e^{-\frac{\pi}{2}\tau}\right\} d\tau<\infty.
\]

\bigskip{}
\bigskip{}

Thus, by Fubini's theorem, it is possible to interchange the orders
of integration and obtain the integral
\begin{align}
F_{\mu,\xi}(x) & :=\intop_{-\infty}^{\infty}W_{\mu,i\tau}(x)\,e^{-i\tau\xi}\,d\tau=e^{x/2}\intop_{-\infty}^{\infty}\,\frac{e^{-i\tau\xi}}{2\pi i}\,\intop_{\sigma-i\infty}^{\sigma+i\infty}\frac{\Gamma\left(\frac{2s-2i\tau+1}{2}\right)\,\Gamma\left(\frac{2s+2i\tau+1}{2}\right)}{\Gamma\left(s-\mu+1\right)}\,x^{-s}ds\,d\tau\nonumber \\
 & =\frac{e^{x/2}}{2\pi i}\intop_{\sigma-i\infty}^{\sigma+i\infty}\,\frac{x^{-s}}{\Gamma(s+1-\mu)}\intop_{-\infty}^{\infty}\,\Gamma\left(\frac{1}{2}+s-i\tau\right)\Gamma\left(\frac{1}{2}+s+i\tau\right)\,e^{-i\tau\xi}\,d\tau\,ds.\label{after the interchange}
\end{align}
The integal with respect to the product of $\Gamma-$function is a
well-known Fourier transform and it is usually given as an application
of Parseval's formula {[}\cite{titchmarsh_fourier_integrals}, p.
105{]}. Indeed, we know from {[}\cite{Yakubovich_Index_Transforms}, p. 1.104{]}
(cf. {[}\cite{MARICHEV}, Vol. II, 2.2.4.5{]}) that the Fourier transform holds
\begin{equation}
\intop_{-\infty}^{\infty}\Gamma(z+it)\,\Gamma(z-it)\,e^{-i\xi t}\,dt=\frac{\sqrt{\pi}\,\Gamma(z)\,\Gamma\left(z+\frac{1}{2}\right)}{\cosh^{2z}\left(\frac{\xi}{2}\right)},\,\,\,\qquad\text{Re}(z)>0,\label{Fourier transform Product gammas}
\end{equation}
and so, making elementary substitutions in (\ref{after the interchange}),
we now obtain the expression for $F_{\mu,\xi}(x)$ in the form
\begin{equation}
F_{\mu,\xi}(x)=\frac{\sqrt{\pi}e^{\frac{x}{2}}}{2\pi i\,\cosh\left(\frac{\xi}{2}\right)}\,\intop_{\sigma-i\infty}^{\sigma+i\infty}\frac{\Gamma\left(s+\frac{1}{2}\right)\,\Gamma\left(s+1\right)}{\Gamma(s+1-\mu)}\,\left(x\cosh^{2}\left(\frac{\xi}{2}\right)\right)^{-s}ds.\label{intermediate fourier}
\end{equation}
In the meantime, recalling entry 8.4.44.1 in {[}\cite{MARICHEV}, Vol.
III{]} or {[}\cite{handbook_marichev}, p. 458, eq. (3.30.2.4){]}, we know that
(\ref{Mellin barnes Whittaker}) can be generalized via the Mellin-Barnes integral
\[
e^{-x/2}W_{\rho,\nu}(x)=\frac{1}{2\pi i}\,\intop_{c-i\infty}^{c+i\infty}\frac{\Gamma\left(s-\nu+\frac{1}{2}\right)\Gamma\left(s+\nu+\frac{1}{2}\right)}{\Gamma\left(s-\rho+1\right)}\,x^{-s}ds,\,\,\,\,x>0,\,\,c>|\text{Re}(\nu)|-\frac{1}{2}.
\]
Thus, replacing $\rho$ by $\frac{1}{4}+\mu$ and $\nu$ by $-\frac{1}{4}$ in the previous representation,
we know that, for any $c>-\frac{1}{4}$ and $x>0$,
\begin{equation}
e^{-x/2}W_{\frac{1}{4}+\mu,-\frac{1}{4}}(x)=\frac{1}{2\pi i}\,\intop_{c-i\infty}^{c+i\infty}\frac{\Gamma\left(s+\frac{3}{4}\right)\Gamma\left(s+\frac{1}{4}\right)}{\Gamma\left(s-\mu+\frac{3}{4}\right)}\,x^{-s}ds=\frac{x^{-\frac{1}{4}}}{2\pi i}\,\intop_{d-i\infty}^{d+i\infty}\frac{\Gamma\left(s+1\right)\Gamma\left(s+\frac{1}{2}\right)}{\Gamma\left(s+1-\mu\right)}\,x^{-s}\,ds,\label{another expression returning}
\end{equation}
where $d:=c-\frac{1}{4}>-\frac{1}{2}$. Since (\ref{another expression returning})
holds for any $d>-\frac{1}{2}$, we can replace $x$ by $x\cosh^{2}(\xi/2)$ there
and use this modification of (\ref{another expression returning}) in (\ref{intermediate fourier})
to produce the formula 
\begin{equation*}
F_{\mu,\xi}(x) =\frac{\sqrt{\pi}x^{\frac{1}{4}}e^{\frac{x}{4}}}{\cosh^{\frac{1}{2}}\left(\frac{\xi}{2}\right)}\,e^{-\frac{x}{4}\cosh\left(\xi\right)}W_{\frac{1}{4}+\mu,-\frac{1}{4}}\left(x\cosh^{2}\left(\frac{\xi}{2}\right)\right)=2^{-\mu}\sqrt{\pi x}\,e^{\frac{x}{4}\left(1-\cosh(\xi)\right)}\,D_{2\mu}\left(\sqrt{2x}\,\cosh\left(\frac{\xi}{2}\right)\right),\label{first expression Fourier transform is paper}
\end{equation*}
where we have invoked the definition of the parabolic cylinder function,
(\ref{cylinder function def}). Since $D_{2\mu}\left(\sqrt{2x}\,\cosh\left(\frac{\xi}{2}\right)\right)\sim\left(2x\right)^{\mu}\cosh^{2\mu}(\xi/2)\,e^{-\frac{x}{2}\cosh^{2}(\xi/2)}$ as $\xi\rightarrow\infty$ (see (\ref{Whittaker asymptotic formula}) above),
we find that
\[
|F_{\mu,\xi}(x)|\sim\sqrt{\pi}\,x^{\mu+\frac{1}{2}}\cosh^{2\mu}\left(\frac{\xi}{2}\right)\,e^{-\frac{x}{2}\cosh(\xi)},\,\,\,\,\,\,\xi\rightarrow\infty. 
\]
When combined with (\ref{expansion index whittaker!}), this asymptotic estimate shows
that, for any given fixed $x>0$, $f(\tau):=W_{\mu,i\tau}(x)\in\mathscr{F}$
and $\hat{f}(\xi):=F_{\mu,\xi}(x)\in\mathscr{F}$, where $\mathscr{F}$
denotes the set of functions having moderate decrease (cf. [\cite{stein_fourier}, pp. 140-142]).\footnote{Or, invoking another version of the Fourier inversion theorem, $f(\tau)$
and $\hat{f}(\xi):=F_{\mu,\xi}(x)$ are both in $L_{1}(\mathbb{R})$.} Thus, we may apply the inversion theorem for the Fourier transform,
which is now given in the following result. 

\begin{lemma} \label{fourier transform Whittaker}
For any $\mu\in\mathbb{R}$ and $x>0$, we have the pair of index
Fourier transforms
\begin{equation}
\intop_{-\infty}^{\infty}W_{\mu,i\tau}(x)\,e^{-i\tau\xi}\,d\tau=2^{-\mu}\sqrt{\pi x}\,e^{\frac{x}{4}\left(1-\cosh(\xi)\right)}\,D_{2\mu}\left(\sqrt{2x}\,\cosh\left(\frac{\xi}{2}\right)\right),\label{Transform Fourier Whittaker at the beginning}
\end{equation}
\begin{equation}
W_{\mu,i\tau}(x)=\sqrt{\frac{x}{\pi}}\,\frac{e^{x/4}}{2^{\mu+1}}\intop_{-\infty}^{\infty}\,e^{-\frac{x}{4}\cosh(\xi)}\,D_{2\mu}\left(\sqrt{2x}\,\cosh\left(\frac{\xi}{2}\right)\right)\,e^{i\tau\xi}\,d\xi.\label{Fourier transform inverse Whittaker}
\end{equation}
\end{lemma}

\bigskip{}

By the reduction formulas (\ref{Whittaker to MAcdonald}) and (\ref{D0 formula!}),
we obtain from (\ref{Fourier transform inverse Whittaker}) the following pair of Fourier transforms
\begin{equation}
\intop_{-\infty}^{\infty}K_{i\tau}\left(x\right)\,e^{-i\tau\xi}\,d\tau=\pi e^{-x\cosh\left(\xi\right)},\,\,\,\,\,K_{i\tau}(x)=\frac{1}{2}\intop_{-\infty}^{\infty}e^{-x\cosh(\xi)}\,e^{i\tau\xi}\,d\xi,\label{first pair}
\end{equation}
which we have already presented at the introduction (see (\ref{Poisson schlafli}) above). 

\bigskip{}

\bigskip{}

After having the Fourier transform of $W_{\mu,i\tau}(x)$, it is now a matter
of using a classical version of the Poisson summation formula to get
an explicit evaluation of the infinite series (\ref{whittaker start}).
We will prove, however, that both $f(\tau):=W_{\mu,i\tau}(x)\in\mathscr{F}$
and $\hat{f}(\xi):=F_{\mu,\xi}(x)\in\mathscr{F}$ belong
to a subclass of $\mathscr{F}$ introduced in [\cite{stein}, pp. 113-114].
This is done in order to give a purely complex analytic proof of
the Poisson summation formula attached to the second index of the Whittaker
function.

\begin{theorem} \label{Main Theorem}
Let $\alpha,x>0$ and $\mu\in\mathbb{R}$. Then the following summation
formula takes place
\begin{equation}
W_{\mu,0}(x)+2\sum_{n=1}^{\infty}W_{\mu,i\alpha n}(x)=\frac{2^{-\mu}\sqrt{\pi x}}{\alpha}\,\left(D_{2\mu}\left(\sqrt{2x}\right)+2e^{\frac{x}{4}}\sum_{n=1}^{\infty}e^{-\frac{x}{4}\cosh\left(\frac{2\pi n}{\alpha}\right)}D_{2\mu}\left(\sqrt{2x}\cosh\left(\frac{\pi n}{\alpha}\right)\right)\right).\label{Theorem 1 summation formula Whittaker}
\end{equation}
\end{theorem}

\begin{proof}
First, for fixed $x\in(0,X]$, we need to establish an estimate for
$W_{\mu,\sigma+i\tau}(x)$ when $\tau$ is large and $\sigma$ is
a fixed real number. Taking the definition of the Whittaker function
as (cf. (\ref{Whittaker definition confluent}) and (\ref{characterization whittaker function}) above)
\begin{align}
W_{\mu,\sigma+it}(x) & =e^{-x/2}x^{\frac{1}{2}+\sigma+i\tau}\left\{ \frac{\Gamma(-2\sigma-2i\tau)}{\Gamma\left(\frac{1}{2}-\sigma-i\tau-\mu\right)}\,_{1}F_{1}\left(\frac{1}{2}+\sigma+i\tau-\mu;\,2\sigma+2i\tau+1;x\right)\right.\nonumber \\
 & \left.+x^{-2\sigma-2i\tau}\,\frac{\Gamma(2\sigma+2i\tau)}{\Gamma\left(\frac{1}{2}+\sigma+i\tau-\mu\right)}\,_{1}F_{1}\left(\frac{1}{2}-\sigma-i\tau-\mu;\,1-2\sigma-2i\tau;\,x\right)\right\} ,\label{Whittaker definition}
\end{align}
we see that, if we write the power series for Kummer's function,
we are able to obtain the asymptotic expansion
\begin{align*}
_{1}F_{1}\left(\frac{1}{2}+\sigma+i\tau-\mu;\,2\sigma+2i\tau+1;x\right) & =\lim_{M\rightarrow\infty}\sum_{k=0}^{M}\frac{\left(\frac{1}{2}+\sigma+i\tau-\mu\right)_{k}}{\left(1+2\sigma+2i\tau\right)_{k}}\,\frac{x^{k}}{k!}\\
=\lim_{M\rightarrow\infty}\sum_{k=0}^{M}\frac{(i\tau)^{k}}{(2i\tau)^{k}}\,\frac{x^{k}}{k!}\left[1+O\left(\frac{1}{\tau}\right)\right] & =e^{\frac{x}{2}}\left[1+O\left(\frac{1}{\tau}\right)\right],\,\,\,\,\tau\rightarrow\infty,
\end{align*}
with an analogous asymptotic result being valid for the second confluent
hypergeometric function in (\ref{Whittaker definition}). Using Stirling's
formula for the rations containing $\Gamma-$functions in (\ref{Whittaker definition}),
we are able to deduce the bound
\begin{equation}
\left|W_{\mu,\sigma+i\tau}(x)\right|\leq\frac{\tau^{\mu-\frac{1}{2}}}{\sqrt{2}}x^{\frac{1}{2}+\sigma}e^{-\frac{\pi\tau}{2}}\left\{ 2^{-2\sigma}\tau^{-\sigma}+2^{2\sigma}\,x^{-2\sigma}\tau^{\sigma}\right\} \left[1+O\left(\frac{1}{\tau}\right)\right],\label{as tau infinity}
\end{equation}
which holds for any $\tau\geq T\gg 1$, $\sigma\in\mathbb{R}$ and $x\in(0,X]$.
Note that, if $\tau\rightarrow-\infty$, the modification of (\ref{as tau infinity})
reads
\begin{equation}
\left|W_{\mu,\sigma+i\tau}(x)\right|\leq\frac{|\tau|^{\mu-\frac{1}{2}}}{\sqrt{2}}x^{\frac{1}{2}+\sigma}e^{-\frac{\pi|\tau|}{2}}\left\{ 2^{-2\sigma}|\tau|^{-\sigma}+2^{2\sigma}\,x^{-2\sigma}|\tau|^{\sigma}\right\} \left[1+O\left(\frac{1}{|\tau|}\right)\right].\label{negative tau estimate}
\end{equation}

We are now ready to prove the summation formula (\ref{Theorem 1 summation formula Whittaker}):
for a fixed $x>0$ and $\mu\in\mathbb{R}$, let $g(z):=W_{\mu,i\alpha z}(x)/\left(e^{2\pi iz}-1\right)$.
Since, for a fixed $x$, $W_{\mu,\nu}(x)$ is an entire function of
the first and second parameters {[}\cite{NIST}, p. 335, 13.14 (ii){]}, we
know that $g(z)$ is analytic everywhere except at the integers, where
it has simple poles with residue $W_{\mu,i\alpha n}(x)/2\pi i$. Thus,
if $b>0$ and $\gamma_{N,b}$ denotes the rectangular contour with
vertices $N+\frac{1}{2}\pm ib$ and $-N-\frac{1}{2}\pm ib$, we have
by the residue Theorem
\begin{equation}
\sum_{|n|\leq N}W_{\mu,i\alpha n}(x)=\intop_{\gamma_{N,b}}\frac{W_{\mu,i\alpha z}(x)}{e^{2\pi iz}-1}dz.\label{formula beginning residue whittaker}
\end{equation}

Note that the contribution of the vertical segments of $\gamma_{N,b}$
tends to zero as $N\rightarrow\infty$. Indeed, by (\ref{as tau infinity})
and (\ref{negative tau estimate}), 
\begin{align*}
\left|\intop_{N+\frac{1}{2}-ib}^{N+\frac{1}{2}+ib}\frac{W_{\mu,i\alpha z}(x)}{e^{2\pi iz}-1}dz\right| & \leq\intop_{-b}^{b}\frac{\left|W_{\mu,-\alpha u+i\alpha\left(N+\frac{1}{2}\right)}(x)\right|}{1+e^{-2\pi u}}\,du\\
\ll b\,N^{\mu-\frac{1}{2}}e^{-\frac{\pi N}{2}}x^{\frac{1}{2}+\alpha b} & \left\{ 2^{2\alpha b}N^{\alpha b}+2^{2\alpha b}\,x^{2\alpha b}N^{\alpha b}\right\} \left[1+O\left(\frac{1}{N}\right)\right],
\end{align*}
which clearly tends to zero as long as $N\rightarrow\infty$. Recalling
(\ref{expansion index whittaker!}), we see that the sum on the left-hand
side of (\ref{formula beginning residue whittaker}) converges to
$\sum_{n\in\mathbb{Z}}W_{\mu,i\alpha n}(x)$. Therefore, we have the
identity
\begin{equation}
\sum_{n\in\mathbb{Z}}W_{\mu,i\alpha n}(x)=\intop_{\Gamma_{1}^{-}(b)}\frac{W_{\mu,i\alpha z}(x)}{e^{2\pi iz}-1}dz-\intop_{\Gamma_{1}^{+}(b)}\frac{W_{\mu,i\alpha z}(x)}{e^{2\pi iz}-1}dz,\label{integral difference!}
\end{equation}
where $\Gamma_{1}^{-}(b)$ and $\Gamma_{1}^{+}(b)$ are the horizontal
lines described by the parametric equation $\gamma_{1}^{\pm}(t):=t\pm ib$,
$t\in\mathbb{R}$. Note that, for $z\in\Gamma_{1}^{-}(b)$, $|e^{2\pi iz}|>1$,
and so we can express the first integral in (\ref{integral difference!}) as 
\begin{align}
\intop_{\Gamma_{1}^{-}(b)}\frac{W_{\mu,i\alpha z}(x)}{e^{2\pi iz}-1}dz & =\sum_{n=1}^{\infty}\,\intop_{-ib-\infty}^{-ib+\infty}W_{\mu,i\alpha z}(x)\,e^{-2\pi inz}dz=\frac{1}{\alpha}\sum_{n=1}^{\infty}\intop_{-\infty}^{\infty}W_{\mu,i\tau}(x)\,e^{-2\pi in\frac{\tau}{\alpha}}d\tau\nonumber \\
 & =\frac{2^{-\mu}}{\alpha}\sqrt{\pi x}\,e^{\frac{x}{4}}\,\sum_{n=1}^{\infty}e^{-\frac{x}{4}\cosh\left(\frac{2\pi n}{\alpha}\right)}\,D_{2\mu}\left(\sqrt{2x}\,\cosh\left(\frac{\pi n}{\alpha}\right)\right),\label{first equation path}
\end{align}
where, in the last step, we have used the evaluation of the Fourier transform (\ref{Transform Fourier Whittaker at the beginning}).
The first equality is justified by absolute convergence, because $\sum_{n=1}^{\infty}e^{-2\pi nb}<\infty$
and $W_{\mu,b\alpha+i\alpha t}(x)$ has the asymptotic behaviors (\ref{as tau infinity})
and (\ref{negative tau estimate}) as $t$ tends to $\infty$ or $-\infty$ respectively. Finally, the second equality in (\ref{first equation path}) comes
from Cauchy's theorem, which can be invoked because $h(z):=W_{\mu,i\alpha z}(x)$
is an analytic function of $z$. Analogously, one can prove that
\begin{equation}
\intop_{\Gamma_{1}^{+}(b)}\frac{W_{\mu,i\alpha z}(x)}{e^{2\pi iz}-1}dz=-\sum_{n=0}^{\infty}\,\intop_{ib-\infty}^{ib+\infty}W_{\mu,i\alpha z}(x)\,e^{2\pi inz}dz=-\frac{2^{-\mu}}{\alpha}\sqrt{\pi x}\,e^{\frac{x}{4}}\,\sum_{n=0}^{\infty}e^{-\frac{x}{4}\cosh\left(\frac{2\pi n}{\alpha}\right)}\,D_{2\mu}\left(\sqrt{2x}\,\cosh\left(\frac{\pi n}{\alpha}\right)\right).\label{second equation path}
\end{equation}
Combining (\ref{first equation path}) and (\ref{second equation path})
gives immediately our formula (\ref{Theorem 1 summation formula Whittaker}).
\end{proof}

\begin{remark}\label{character remark Whittaker}
Since our proof of (\ref{Theorem 1 summation formula Whittaker}) only used the classical Poisson summation formula,
it is natural to consider character analogues of (\ref{Theorem 1 summation formula Whittaker}) (see \cite{koshliakov_ramanujan_character, berndt_character_poisson} for character analogues of some classical summation
formulas). In fact, it is not hard to show that, if $\chi$ is a nonprincipal,
primitive and even Dirichlet character modulo $q$, then the following
summation formula holds
\[
\sum_{n=1}^{\infty}\chi(n)\,W_{\mu,i\alpha n}(x)=\frac{2^{-\mu}\sqrt{\pi x}\,G(\chi)e^{\frac{x}{4}}}{q\alpha}\,\sum_{n=1}^{\infty}\overline{\chi}(n)\,e^{-\frac{x}{4}\cosh\left(\frac{2\pi n}{q\alpha}\right)}\,D_{2\mu}\left(\sqrt{2x}\,\cosh\left(\frac{\pi n}{q\alpha}\right)\right),
\]
where $G(\chi)$ denotes the Gauss sum,
\[
G(\chi):=\sum_{r=1}^{q-1}\chi(r)\,e^{2\pi ir/q}.
\]
Analogously, if $\chi$ is a nonprincipal, primitive and odd Dirichlet
character modulo $q$, then the summation formula holds
\[
\sum_{n=1}^{\infty}\chi(n)\,n\,W_{\mu,i\alpha n}(x)=-\frac{i2^{-\mu-\frac{1}{2}}\sqrt{\pi}\,x\,G(\chi)\,e^{\frac{x}{4}}}{q\alpha}\,\sum_{n=1}^{\infty}\overline{\chi}(n)\,e^{-\frac{x}{4}\cosh\left(\frac{2\pi n}{q\alpha}\right)}\sinh\left(\frac{\pi n}{q\alpha}\right)\,D_{2\mu+1}\left(\sqrt{2x}\,\cosh\left(\frac{\pi n}{q\alpha}\right)\right).
\]

Although these character analogues are valid, we do not know how to
extend the scope of Theorem \ref{Main Theorem} to the class of Dirichlet series satisfying
Hecke's functional equation. For example, we do not know which kind
of transformation formulas one can get while studying the infinite series
\[
\sum_{n=1}^{\infty}r_{k}(n)\,W_{\mu,i\sqrt{n}}(x),\,\,\,\,\sum_{n=1}^{\infty}a_{f}(n)\,W_{\mu,i\sqrt{n}}(x),\,\,\,\,\mu\in\mathbb{R},\,\,\,x>0,
\]
where $r_{k}(n)$ denotes the number of ways in which we can write
$n$ as a sum of $k$ squares and $a_{f}(n)$ denotes the $n^{\text{th}}$
Fourier coefficient of a holomorphic cusp form $f(z)$ for the full modular group. 

\end{remark}

\bigskip{}

\bigskip{}

Using the reduction formulas (\ref{Whittaker to MAcdonald}) and (\ref{D0 formula!})
(or, alternatively, applying the Poisson summation formula to the
sum $K_{i\alpha n}(x)$ and using (\ref{first pair})), the previous formula gives
a summation formula with respect to the index of the Macdonald function.
The following corollary was given for the first time in \cite{Yakubovich_Lebedev}
and it is due to the second author of this article.

\begin{corollary} \label{corollary 1}
For any $x,\alpha>0$, the following identity takes place
\begin{equation}
K_{0}(x)+2\,\sum_{n=1}^{\infty}K_{in\alpha}(x)=\frac{\pi}{\alpha}e^{-x}+\frac{2\pi}{\alpha}\sum_{n=1}^{\infty}e^{-x\cosh\left(\frac{2\pi n}{\alpha}\right)}.\label{semyon formula Kin}
\end{equation}

\end{corollary}

\bigskip{}
\bigskip{}

Some immediate consequences of (\ref{semyon formula Kin}) are already
provided in \cite{Yakubovich_Lebedev}. The most important of these is the transformation
formula {[}\cite{Yakubovich_Lebedev}, p. 299, formula (24){]} 
\begin{equation}
\sum_{n\in\mathbb{Z}}\Gamma\left(\frac{s}{2}+i\alpha n\right)\Gamma\left(\frac{s}{2}-i\alpha n\right)=\frac{\sqrt{\pi}}{\alpha}\,\Gamma\left(\frac{s}{2}\right)\Gamma\left(\frac{s+1}{2}\right)\,\sum_{n\in\mathbb{Z}}\frac{1}{\cosh^{s}\left(\frac{\pi n}{\alpha}\right)},\,\,\,\,\,\text{Re}(s)>0,\,\,\alpha>0,\label{Summation formula product gammas}
\end{equation}
which can be alternatively deduced from (\ref{Fourier transform Product gammas})
and the Poisson summation formula.

\bigskip{}

In the remainder of our paper we derive several corollaries of the
formulas (\ref{Theorem 1 summation formula Whittaker}) and (\ref{semyon formula Kin}),
which seem to represent new formulas. We start with a summation formula
over the imaginary indices of two distinct products of Macdonald functions. 

\begin{corollary} \label{corollary 2}

Let $x,y,\alpha>0$. Then the following formula holds
\begin{equation}
K_{0}(x)\,K_{0}(y)+2\sum_{n=1}^{\infty}K_{i\alpha n}\left(x\right)K_{i\alpha n}(y)=\frac{\pi}{\alpha}K_{0}\left(x+y\right)+\frac{2\pi}{\alpha}\sum_{n=1}^{\infty}K_{0}\left(\sqrt{x^{2}+y^{2}+2xy\cosh\left(\frac{2\pi n}{\alpha}\right)}\right).\label{product 1}
\end{equation}
Moreover, if $\sigma\in\mathbb{R}$ and $x,\alpha>0$, then the following
summation formula takes place
\begin{equation}
K_{\sigma}^{2}(x)+2\sum_{n=1}^{\infty}K_{\sigma+i\alpha n}(x)\,K_{\sigma-i\alpha n}(x)=\frac{\pi}{\alpha}K_{2\sigma}(2x)+\frac{2\pi}{\alpha}\sum_{n=1}^{\infty}K_{2\sigma}\left(2x\cosh\left(\frac{\pi n}{\alpha}\right)\right).\label{product 2}
\end{equation}

\end{corollary}

\begin{proof}

The first formula can be proved by appealing to relation 2.16.9.1
on page 354 of the second volume of \cite{MARICHEV}, which states that
\[
K_{i\tau}(x)\,K_{i\tau}(y)=\frac{1}{2}\intop_{0}^{\infty}\exp\left(-\frac{1}{2}\left(\frac{x^{2}+y^{2}}{xy}u+\frac{xy}{u}\right)\right)\,K_{i\tau}(u)\,\frac{du}{u},
\]
for $x,y>0$. Using this product formula and Corollary \ref{corollary 1}, we find
that 
\begin{align}
\sum_{n\in\mathbb{Z}}K_{i\alpha n}(x)\,K_{i\alpha n}(y) & =\frac{1}{2}\intop_{0}^{\infty}\exp\left(-\frac{1}{2}\left(\frac{x^{2}+y^{2}}{xy}u+\frac{xy}{u}\right)\right)\,\sum_{n\in\mathbb{Z}}K_{i\alpha n}(u)\,\frac{du}{u}\nonumber\\
 & = \frac{\pi}{2\alpha}\,\intop_{0}^{\infty}\exp\left(-\frac{1}{2}\left(\frac{x^{2}+y^{2}}{xy}u+\frac{xy}{u}\right)\right)\,\sum_{n\in\mathbb{Z}}e^{-u\cosh\left(\frac{2\pi n}{\alpha}\right)}\,\frac{du}{u}\nonumber\\
 & =\frac{\pi}{2\alpha}\sum_{n\in\mathbb{Z}}\intop_{0}^{\infty}\exp\left(-\frac{1}{2}\left(\frac{x^{2}+y^{2}+2xy\cosh\left(\frac{2\pi n}{\alpha}\right)}{xy}u+\frac{xy}{u}\right)\right)\,\frac{du}{u}\nonumber\\
 & =\frac{\pi}{\alpha}\sum_{n\in\mathbb{Z}}K_{0}\left(\sqrt{x^{2}+y^{2}+2xy\cosh\left(\frac{2\pi n}{\alpha}\right)}\right), \label{justification product formula}
\end{align}
where the interchange of the orders of integration and summation in
both steps is justified by absolute convergence and Fubini's theorem.
Indeed, since $|K_{i\alpha n}(u)|\leq K_{0}\left(\frac{u}{2}\right)e^{-\frac{\pi}{3}\alpha|n|}$ (see (\ref{uniform estimate Macdonald index and X}) above), we know that 
\begin{equation*}
\sum_{n\in\mathbb{Z}}\intop_{0}^{\infty}\exp\left(-\frac{1}{2}\left(\frac{x^{2}+y^{2}}{xy}u+\frac{xy}{u}\right)\right)|K_{i\alpha n}(u)|\,du\leq\sum_{n\in\mathbb{Z}}e^{-\frac{\pi}{3}\alpha|n|}\intop_{0}^{\infty}\exp\left(-\frac{1}{2}\left(\frac{x^{2}+y^{2}}{xy}u+\frac{xy}{u}\right)\right)K_{0}\left(\frac{u}{2}\right)\,du,
\end{equation*}
with the integral being convergent due to (\ref{asymptotic K}) and (\ref{asymptotic small equal to zero}). Analogously, the third equality in (\ref{justification product formula}) is justified by the following steps
\begin{align}
\sum_{n=1}^{\infty}\intop_{0}^{\infty}\exp\left(-\frac{1}{2}\left(\frac{x^{2}+y^{2}}{xy}u+\frac{xy}{u}\right)\right)\,e^{-u\cosh\left(\frac{2\pi n}{\alpha}\right)}\,\frac{du}{u}\nonumber \\
=\sum_{n=1}^{\infty}\intop_{0}^{\infty}\exp\left(-\frac{1}{2}\left(\frac{x^{2}+y^{2}}{xy}u+\frac{xy}{u}\right)\right)\,e^{-u-2u\sinh^{2}\left(\frac{\pi n}{\alpha}\right)}\,\frac{du}{u}\nonumber \\
\leq\sum_{n=1}^{\infty}\intop_{0}^{\infty}\exp\left(-\frac{1}{2}\left(\frac{x^{2}+y^{2}}{xy}u+\frac{xy}{u}\right)\right)\,\frac{e^{-u}}{1+2u\sinh^{2}\left(\frac{\pi n}{\alpha}\right)}\,\frac{du}{u}\nonumber \\
\leq\frac{1}{2\sqrt{2}}\sum_{n=1}^{\infty}\frac{1}{\sinh\left(\frac{\pi n}{\alpha}\right)}\intop_{0}^{\infty}\exp\left(-\frac{1}{2}\left(\frac{x^{2}+y^{2}+2xy}{xy}u+\frac{xy}{u}\right)\right)\,\frac{du}{u^{\frac{3}{2}}}<\infty.\label{steps second justification template}
\end{align}
Finally, in the last equality of (\ref{justification product formula}) we have used the integral representation [\cite{NIST}, p. 253, 10.32.10],
\begin{equation}
K_{\nu}(x)=\frac{1}{2}\left(\frac{x}{2}\right)^{\nu}\intop_{0}^{\infty}\exp\left(-u-\frac{x^{2}}{4u}\right)\,\frac{du}{u^{\nu+1}},\,\,\,\,\,x>0,\,\,\,\nu\in\mathbb{C}.\label{representation macdonald as convolution Mellin}
\end{equation}

\bigskip{}

To prove the second formula (\ref{product 2}), we use the following representation given
in {[}\cite{MARICHEV}, p. 349, relation 2.16.5.4{]},
\[
\intop_{0}^{\infty}u^{2\sigma-1}K_{2i\tau}\left(xu+\frac{x}{u}\right)du=K_{\sigma+i\tau}(x)\,K_{\sigma-i\tau}(x),
\]
which is valid for any $x>0$ and real $\sigma$. Using the appropriate
substitutions, we can write the left-hand side of (\ref{product 2}) as
\begin{equation}
\sum_{n\in\mathbb{Z}}K_{\sigma+i\alpha n}(x)\,K_{\sigma-i\alpha n}(x)=\sum_{n\in\mathbb{Z}}\intop_{0}^{\infty}u^{2\sigma-1}K_{2i\alpha n}\left(xu+\frac{x}{u}\right)du.\label{Expression after product}
\end{equation}
Now, since $|K_{i\tau}(x)|\leq e^{-\delta|\tau|}K_{0}(x\cos(\delta))$
for any $\delta\in[0,\frac{\pi}{2})$, we know that
\begin{align*}
\sum_{n\in\mathbb{Z}}\intop_{0}^{\infty}u^{2\sigma-1}\left|K_{2i\alpha n}\left(xu+\frac{x}{u}\right)\right|du & \leq\sum_{n\in\mathbb{Z}}e^{-2\alpha|n|\delta}\intop_{0}^{\infty}u^{2\sigma-1}K_{0}\left(x\cos(\delta)\left(u+\frac{1}{u}\right)\right)\,du\\
 & =\frac{1+e^{-2\alpha\delta}}{1-e^{-2\alpha\delta}}\,K_{\sigma}^{2}\left(x\cos(\delta)\right),
\end{align*}
which proves that the interchange of the orders of summation and integration is possible in (\ref{Expression after product}) and we have
\begin{align}
\intop_{0}^{\infty}u^{2\sigma-1}\sum_{n\in\mathbb{Z}}K_{2i\alpha n}\left(xu+\frac{x}{u}\right)du & =\frac{\pi}{2\alpha}\,\intop_{0}^{\infty}u^{2\sigma-1}\sum_{n\in\mathbb{Z}}e^{-\left(xu+\frac{x}{u}\right)\cosh\left(\frac{\pi n}{\alpha}\right)}\,du\nonumber \\
=\frac{\pi}{2\alpha}\sum_{n\in\mathbb{Z}}\intop_{0}^{\infty}u^{2\sigma-1}e^{-x\cosh\left(\frac{\pi n}{\alpha}\right)\left(u+\frac{1}{u}\right)}du & =\frac{\pi}{\alpha}\,\sum_{n\in\mathbb{Z}}K_{2\sigma}\left(2x\cosh\left(\frac{\pi n}{\alpha}\right)\right),\label{step final proof second product Macdonald}
\end{align}
which completes our proof. The last equality is clearly a direct consequence of
(\ref{representation macdonald as convolution Mellin}) and the relation
$K_{\nu}(z)=K_{-\nu}(z)$. The second equality in (\ref{step final proof second product Macdonald})
is again justified because
\begin{align*}
\sum_{n=1}^{\infty}\intop_{0}^{\infty}u^{2\sigma-1}e^{-x\left(u+\frac{1}{u}\right)\cosh\left(\frac{\pi n}{\alpha}\right)}du & =\sum_{n=1}^{\infty}\intop_{0}^{\infty}u^{2\sigma-1}e^{-x\left(u+\frac{1}{u}\right)\left(1+2\sinh^{2}\left(\frac{\pi n}{2\alpha}\right)\right)}du\\
\leq\sum_{n=1}^{\infty}\intop_{0}^{\infty}\frac{u^{2\sigma-1}\,e^{-xu+\frac{x}{u}}}{1+2x\left(u+\frac{1}{u}\right)\sinh^{2}\left(\frac{\pi n}{2\alpha}\right)}du & \leq\frac{1}{2\sqrt{2x}}\sum_{n=1}^{\infty}\frac{1}{\sinh\left(\frac{\pi n}{2\alpha}\right)}\,\intop_{0}^{\infty}\frac{u^{2\sigma-\frac{1}{2}}e^{-xu+\frac{x}{u}}}{\sqrt{1+u^{2}}}\,du
\end{align*}
and the previous integral is convergent for any real $\sigma$.
\end{proof}

\bigskip{}

Our next result provides two elegant transformation formulas, which
seem to be unnoticed in the literature.

\begin{corollary}\label{corollary 3}
Let $\alpha>0$ and $0<\varphi<\frac{\pi}{2}$. Then the following
summation formulas take place
\begin{equation}
\frac{2\varphi}{\pi}+2\sum_{n=1}^{\infty}\frac{\sinh\left(\alpha\varphi\,n\right)}{\sinh\left(\frac{\pi\alpha n}{2}\right)}=\frac{2\sin(\varphi)}{\alpha\cos(\varphi)}+\frac{2\sin(2\varphi)}{\alpha}\sum_{n=1}^{\infty}\frac{1}{\cosh^{2}\left(\frac{2\pi n}{\alpha}\right)-\sin^{2}(\varphi)},\label{first elementary formula index}
\end{equation}
\begin{equation}
1+2\sum_{n=1}^{\infty}\frac{\cosh\left(\alpha\varphi\,n\right)}{\cosh\left(\frac{\pi\alpha n}{2}\right)}=\frac{2}{\alpha\cos\varphi}+\frac{4\cos(\varphi)}{\alpha}\,\sum_{n=1}^{\infty}\frac{\cosh\left(\frac{2\pi n}{\alpha}\right)}{\cosh^{2}\left(\frac{2\pi n}{\alpha}\right)-\sin^{2}\varphi}.\label{second elementary formula index}
\end{equation}
\end{corollary}

\begin{proof}
Invoking the integral representation given in {[}\cite{MARICHEV}, Vol. II, p.
356, relation 2.16.11.1{]}, we know that
\begin{equation}
\intop_{0}^{\infty}\sinh(ru)\,K_{i\tau}(u)\,du=\frac{\pi}{2\sqrt{1-r^{2}}}\frac{\sinh\left(\tau\,\arcsin\left(r\right)\right)}{\sinh\left(\frac{\pi\tau}{2}\right)},\,\,\,\,0<r<1.\label{sinh integral in the proof}
\end{equation}
Thus, if we use the summation formula (\ref{semyon formula Kin}),
we obtain
\begin{equation}
\frac{\pi}{2\sqrt{1-r^{2}}}\sum_{n\in\mathbb{Z}}\frac{\sinh\left(\alpha\arcsin(r)\,n\right)}{\sinh\left(\frac{\pi\alpha n}{2}\right)}=\intop_{0}^{\infty}\sinh(ru)\,\sum_{n\in\mathbb{Z}}K_{in\alpha}(u)\,du,\label{first equality in the proof corollary 3}
\end{equation}
where the interchange of the integral and the series is justified by
the fact that $|K_{in\alpha}(u)|\leq K_{0}\left(u\,\cos(\delta)\right)e^{-\alpha\delta|n|},$
$\delta\in[0,\frac{\pi}{2})$ (cf. (\ref{uniform estimate Macdonald index and X}) above). Picking $\delta$ so small that $\cos(\delta)>r$,
we see that
\begin{align*}
\sum_{n\in\mathbb{Z}}\intop_{0}^{\infty}\sinh(ru)|K_{in\alpha}(u)|\,du & \leq\sum_{n\in\mathbb{Z}}e^{-\alpha\delta|n|}\intop_{0}^{\infty}e^{ru}K_{0}\left(u\cos(\delta)\right)\,du\\
 & \ll\frac{1+e^{-\alpha\delta}}{1-e^{-\alpha\delta}}\sqrt{\frac{\pi}{2\cos(\delta)}}\,e^{r-\cos(\delta)}<\infty,
\end{align*}
which justifies the equality (\ref{first equality in the proof corollary 3}) (note that the final estimate is just a consequence of (\ref{asymptotic K})).
Next, using the transformation formula (\ref{semyon formula Kin})
on the right-hand side of (\ref{first equality in the proof corollary 3}),
we obtain
\begin{align}
\frac{\pi}{2\sqrt{1-r^{2}}}\sum_{n\in\mathbb{Z}}\frac{\sinh\left(\alpha\arcsin(r)\,n\right)}{\sinh\left(\frac{\pi\alpha n}{2}\right)} & =\frac{\pi}{\alpha}\intop_{0}^{\infty}\sinh(ru)\,\sum_{n\in\mathbb{Z}}e^{-u\cosh\left(\frac{2\pi n}{\alpha}\right)}\,du\nonumber \\
=\frac{\pi}{\alpha}\sum_{n\in\mathbb{Z}}\intop_{0}^{\infty}\sinh(ru)\,e^{-u\cosh\left(\frac{2\pi n}{\alpha}\right)}du & =\frac{\pi r}{\alpha}\,\sum_{n\in\mathbb{Z}}\frac{1}{\cosh^{2}\left(\frac{2\pi n}{\alpha}\right)-r^{2}},\label{after Laplace transform}
\end{align}
where the interchange performed in the second equality can be
immediately justified by the following steps
\begin{align*}
\sum_{n=1}^{\infty}\intop_{0}^{\infty}\sinh(ru)\,e^{-u\cosh\left(\frac{2\pi n}{\alpha}\right)}du & =\sum_{n=1}^{\infty}\intop_{0}^{\infty}\sinh(ru)\,e^{-u-2u\sinh^{2}\left(\frac{\pi n}{\alpha}\right)}du\\
\leq\sum_{n=1}^{\infty}\intop_{0}^{\infty}\sinh(ru)\,\frac{e^{-u}}{1+2u\sinh^{2}\left(\frac{\pi n}{\alpha}\right)}du & \leq\frac{1}{2\sqrt{2}}\,\sum_{n=1}^{\infty}\frac{1}{\sinh\left(\frac{\pi n}{\alpha}\right)}\,\intop_{0}^{\infty}\frac{e^{-u}}{\sqrt{u}}\sinh(ru)\,du<\infty,
\end{align*}
because $0<r<1$ by hypothesis.
Finally, taking $r=\sin(\varphi)$
with $0<\varphi<\frac{\pi}{2}$, and comparing (\ref{first equality in the proof corollary 3})
with (\ref{after Laplace transform}), we are able to deduce
\[
\frac{2\varphi}{\pi}+2\sum_{n=1}^{\infty}\frac{\sinh\left(\alpha\varphi\,n\right)}{\sinh\left(\frac{\pi\alpha n}{2}\right)}=\frac{2\sin(\varphi)}{\alpha\cos(\varphi)}+\frac{2\sin(2\varphi)}{\alpha}\sum_{n=1}^{\infty}\frac{1}{\cosh^{2}\left(\frac{2\pi n}{\alpha}\right)-\sin^{2}(\varphi)},
\]
which is (\ref{first elementary formula index}). 

The proof of (\ref{second elementary formula index})
is analogous, but instead of starting with (\ref{sinh integral in the proof}),
our point of departure is the second entry on relation 2.16.11.1 of
{[}\cite{MARICHEV}, Vol. II, p. 356{]}, 
\[
\intop_{0}^{\infty}\cosh(ru)\,K_{i\tau}(u)\,du=\frac{\pi}{2\sqrt{1-r^{2}}}\frac{\cosh\left(\tau\,\arcsin\left(r\right)\right)}{\cosh\left(\frac{\pi\tau}{2}\right)},\,\,\,\,0<r<1.
\]
\end{proof}

\begin{corollary} \label{corollary 4}
For any $x,\alpha>0$, the following summation formula holds
\begin{equation}
\text{E}_{1}(x)+\frac{2}{\alpha}\sqrt{\frac{x}{\pi}}\,e^{-\frac{x}{2}}\,\sum_{n=1}^{\infty}\frac{\text{Im}\left\{ K_{\frac{1}{2}+i\alpha n}\left(\frac{x}{2}\right)\right\} }{n}
=\frac{\pi}{\alpha}\,\left(\text{erfc}\left(\sqrt{x}\right)+2\sum_{n=1}^{\infty}\text{erfc}\left(\sqrt{x}\cosh\left(\frac{\pi n}{\alpha}\right)\right)\right),\label{Exponential integral summation}
\end{equation}
where $\text{erfc}(\cdot)$ denotes the complementary error function
(\ref{Definition complementary error function!}) and $\text{E}_{1}(\cdot)$
is the exponential integral function {[}\cite{NIST}, p. 150, 6.2.1{]}, usually defined
by
\[
\text{E}_{1}(x):=\intop_{x}^{\infty}\frac{e^{-t}}{t}dt,\,\,\,\,x>0.
\]
\end{corollary}

\begin{proof}
First, let us briefly check that both series of our formula (\ref{Exponential integral summation}) are
absolutely convergent. The convergence of the series on the right-hand
side is immediate by the asymptotic behavior of the error function
(\ref{asymptotic erfc infinity}). The convergence of the series on the left-hand side can be justified
if we appeal to the asymptotic formula (\ref{as tau infinity}) with $\mu=0,\,\sigma=\frac{1}{2}$ and then use (\ref{Whittaker to MAcdonald}). Alternatively,
we may invoke a generalization of (\ref{uniform estimate Macdonald index and X}) that can be found in {[}\cite{Yakubovich_testing}, p. 280, eq.
(1.8){]},
\[
\left|K_{\sigma+i\tau}(x)\right|\leq e^{-\delta|\tau|}\,K_{\sigma}\left(x\cos(\delta)\right),\,\,\,\,\sigma\in\mathbb{R},\,\,\,0<\delta<\frac{\pi}{2},
\]
to obtain the bound
\[
\sum_{n=1}^{\infty}\frac{\left|\text{Im}\left\{ K_{\frac{1}{2}+i\alpha n}\left(\frac{x}{2}\right)\right\} \right|}{n}\leq\sqrt{\frac{\pi}{x\cos(\delta)}}\,e^{-\frac{x}{2}\cos(\delta)}\,\log\left(1-e^{-\alpha\delta}\right),
\]
valid for any $\delta\in(0,\frac{\pi}{2})$.

Having assured convergence, let us now take $\mu=-\frac{1}{2}$ in our formula (\ref{Theorem 1 summation formula Whittaker})
and use the fact that $D_{-1}(z)$ is related to the error function
by (\ref{D-1 formula!}). We obtain a transformation formula of the
form
\begin{equation}
W_{-\frac{1}{2},0}(x)+2\sum_{n=1}^{\infty}W_{-\frac{1}{2},i\alpha n}(x)=\frac{\pi\sqrt{x}}{\alpha}e^{\frac{x}{2}}\,\left(\text{erfc}\left(\sqrt{x}\right)+2\sum_{n=1}^{\infty}\text{erfc}\left(\sqrt{x}\cosh\left(\frac{\pi n}{\alpha}\right)\right)\right).\label{starting point of imposition of first index}
\end{equation}
However, if we now appeal to relation 13.15.10 in [\cite{NIST}, p. 336] (see (\ref{Recurrence Whittaker!}) above),
\[
2\nu W_{\mu,\nu}(z)=\sqrt{z}\left(W_{\mu+\frac{1}{2},\nu+\frac{1}{2}}(z)-W_{\mu+\frac{1}{2},\frac{1}{2}-\nu}(z)\right),
\]
we find that, if $\mu=-\frac{1}{2}$, $\nu=i\tau$, $\tau\neq0$ and
$x>0$, 
\begin{align}
W_{-\frac{1}{2},i\tau}(x) & =\frac{\sqrt{x}}{2i\tau}\left(W_{0,\frac{1}{2}+i\tau}(x)-W_{0,\frac{1}{2}-i\tau}(x)\right)=\frac{\sqrt{x}}{\tau}\,\sqrt{\frac{x}{\pi}}\,\frac{1}{2i}\left(K_{\frac{1}{2}+i\tau}\left(\frac{x}{2}\right)-K_{\frac{1}{2}-i\tau}\left(\frac{x}{2}\right)\right)\nonumber \\
 & =\frac{x}{\tau\sqrt{\pi}}\,\text{Im}\left\{ K_{\frac{1}{2}+i\tau}\left(\frac{x}{2}\right)\right\} ,\label{imaginary thing}
\end{align}
while [\cite{NIST}, p. 254, 10.38.7], 
\begin{equation}
W_{-\frac{1}{2},0}(x) =\frac{x}{\sqrt{\pi}}\,\lim_{\tau\rightarrow0}\frac{K_{\frac{1}{2}+i\tau}\left(\frac{x}{2}\right)-K_{\frac{1}{2}-i\tau}\left(\frac{x}{2}\right)}{2i\tau}=\frac{x}{\sqrt{\pi}}\left\{ \frac{\partial K_{\nu}\left(\frac{x}{2}\right)}{\partial\nu}\right\} _{\nu=\frac{1}{2}}=\sqrt{x}\,\text{E}_{1}(x)\,e^{\frac{x}{2}}.\label{limiting case}
\end{equation}
Combining (\ref{starting point of imposition of first index}) with
the reduction formulas (\ref{imaginary thing}) and (\ref{limiting case})
immediately gives (\ref{Exponential integral summation}).
\end{proof}
\begin{remark}\label{remark 1}
The formula derived in Corollary \ref{corollary 4} can be also established via the
integral representation (see relation 2.16.7.3 in [\cite{MARICHEV},
Vol. II, page 351]),
\begin{equation}
\intop_{1}^{\infty}u^{-\beta-\frac{1}{2}}(u-1)^{\beta-1}e^{-xu}K_{\nu}(xu)\,du=\sqrt{\frac{\pi}{2x}}\Gamma(\beta)\,e^{-x}W_{-\beta,\nu}(2x),\label{representation marimarimarichev}
\end{equation}
which is valid when $x>0$ and $\text{Re}(\beta)>0$. For example,
if we take $\nu=i\tau$, $\tau\in\mathbb{R}\setminus\{0\}$, and $\beta=\frac{1}{2}$
in (\ref{representation marimarimarichev}), we can see that 
\begin{equation}
\intop_{1}^{\infty}\frac{e^{-xu}}{u\sqrt{u-1}}K_{i\tau}(xu)\,du=\frac{\pi e^{-x}}{\sqrt{2x}}\,W_{-\frac{1}{2},i\tau}(2x)=\sqrt{2\pi x}\,\frac{e^{-x}}{\tau}\,\text{Im}\left\{ K_{\frac{1}{2}+i\tau}(x)\right\} ,\label{representation at another proof}
\end{equation}
while, if $\tau=0$,
\[
\intop_{1}^{\infty}\frac{e^{-xu}}{u\sqrt{u-1}}K_{0}(xu)\,du=\frac{\pi e^{-x}}{\sqrt{2x}}\,W_{-\frac{1}{2},0}(2x)=\pi\,\text{E}_{1}(2x).
\]

Taking $\tau=n\alpha$, $n\in\mathbb{Z}\setminus\{0\}$, in (\ref{representation at another proof})
and summing the resulting expression with respect to $n$ (and, of
course, using again the usual justifications for interchanging the
series over $n$ with the integral on the left of (\ref{representation marimarimarichev})),
we find that
\begin{align}
\pi\,\text{E}_{1}(2x)+2\sqrt{2\pi x}\,\frac{e^{-x}}{\alpha}\,\sum_{n=1}^{\infty}\frac{\text{\text{Im}\ensuremath{\left\{  K_{\frac{1}{2}+in\alpha}(x)\right\} } }}{n}\nonumber \\
=\frac{\pi}{\alpha}\sum_{n\in\mathbb{Z}}\intop_{1}^{\infty}\frac{e^{-xu\left(1+\cosh\left(\frac{2\pi n}{\alpha}\right)\right)}}{u\sqrt{u-1}}\,du=\frac{\pi^{2}}{\alpha}\sum_{n\in\mathbb{Z}}\text{erfc}\left(\sqrt{2x}\,\cosh\left(\frac{\pi n}{\alpha}\right)\right),\label{we find that}
\end{align}
which is equivalent to (\ref{Exponential integral summation}). 

In the last equality of (\ref{we find that}), we have invoked the relation
2.3.6.6 on page 324 of [\cite{MARICHEV}, Vol. I],
\begin{equation}
\intop_{1}^{\infty}u^{\alpha-1}(u-1)^{\beta-1}e^{-xu}du=\Gamma(\beta)e^{-x}\Psi\left(\beta,\alpha+\beta;x\right),\,\,\,\,\text{Re}(\beta)>0,\,\,x>0. \label{entry with Psi}
\end{equation}
Taking $\alpha=0$, $\beta=\frac{1}{2}$ and invoking the reduction
formula for the Tricomi function {[}\cite{MARICHEV}, Vol. III, page 584,
relation 7.11.4.7{]},
\[
\Psi\left(\frac{1}{2},\frac{1}{2};x\right)=\sqrt{2}e^{x/2}D_{-1}\left(\sqrt{2x}\right)=\sqrt{\pi}\,e^{x}\text{erfc}\left(\sqrt{x}\right),
\]
we see that (\ref{entry with Psi}) implies the desired kernel
\[
\intop_{1}^{\infty}\frac{e^{-xu\left(1+\cosh\left(\frac{2\pi n}{\alpha}\right)\right)}}{u\sqrt{u-1}}\,du=\pi\,\text{erfc}\left(\sqrt{x}\right).
\]

It is also curious to mention that we can prove our main summation
formula (\ref{Theorem 1 summation formula Whittaker}) from (\ref{semyon formula Kin})
and the integral representation (\ref{representation marimarimarichev}).
Since (\ref{representation marimarimarichev}) is only valid for particular
values of $\beta$ (even when this parameter is assumed to be real),
an argument by analytic continuation is needed at the end of this
second proof. In any case, this alternative reasoning shows that the
summation formulas (\ref{Theorem 1 summation formula Whittaker})
and (\ref{semyon formula Kin}) are, in fact, equivalent.
\end{remark}

\bigskip{}

\bigskip{}

The previous corollaries furnished formulas connected with the classical
theory of the Kontorovich-Lebedev transform. In the next result we
prove a summation formula attached to the Olevskii transform \cite{Yakubovich_Olevskii}. 

\begin{corollary}\label{corollary 5}

Assume that $\nu,x$ and $\alpha$ are positive real numbers and $\mu$ is such that  $\mu>\frac{\nu}{2} $.
Then the following transformation formula holds
\begin{align}
\sum_{n\in\mathbb{Z}}|\Gamma(\mu+i\alpha n)|^{2}\,{}_{2}F_{1}\left(\mu+i\alpha n,\mu-i\alpha n;\,\nu;-x^{2}\right)\nonumber \\
=\frac{\pi}{\alpha}\,2^{2-2\mu}\Gamma\left(2\mu\right)\sum_{n\in\mathbb{Z}}\frac{1}{\cosh^{2\mu}\left(\frac{\pi n}{\alpha}\right)}\,_{2}F_{1}\left(\mu,\mu+\frac{1}{2};\,\nu;\,-\frac{x^{2}}{\cosh^{2}\left(\frac{\pi n}{\alpha}\right)}\right).\label{summation formula Olevskii}
\end{align}

\end{corollary}

\begin{proof}
Before proving (\ref{summation formula Olevskii}) directly, let us
observe that both series represented in the formula are absolutely
convergent. Checking this is immediate for the series on the right-hand
side, because $\mu>0$ and $_{2}F_{1}\left(\mu,\mu+\frac{1}{2};\nu;z\right)\rightarrow1$
as $z\rightarrow0$.

To justify the convergence of the series at the left of (\ref{summation formula Olevskii}),
as well as to prove (\ref{summation formula Olevskii}) itself, we
start by invoking the following representation {[}\cite{MARICHEV}, Vol. II,
relation 2.16.21.1, page 365{]},
\[
\left|\Gamma\left(\mu+i\tau\right)\right|^{2}\,_{2}F_{1}\left(\mu+i\tau,\mu-i\tau;\,\nu;\,-x^{2}\right)=2^{1-2\mu+\nu}x^{1-\nu}\Gamma(\nu)\,\intop_{0}^{\infty}u^{2\mu-\nu}J_{\nu-1}(xu)\,K_{2i\tau}(u)\,du,
\]
which is valid for $\mu,\nu,x>0$. Since $|K_{2i\tau}(u)|\leq K_{0}\left(u\cos(\delta)\right)\,e^{-2\delta|\tau|}$,
$\delta\in[0,\frac{\pi}{2})$, and $|J_{\nu-1}(xu)|\leq\frac{c_{\nu-1}}{\sqrt{x}}$,
we find that the terms of the series on the left-hand side of (\ref{summation formula Olevskii}) have
the behavior 
\begin{align}
\left|\Gamma\left(\mu+i\tau\right)\right|^{2}\,\left|_{2}F_{1}\left(\mu+i\tau,\mu-i\tau;\,\nu;\,-x^{2}\right)\right|\nonumber \\
\leq c_{\nu-1}\,2^{1-2\mu+\nu}x^{\frac{1}{2}-\nu}\Gamma(\nu)e^{-2\delta|\tau|}\,\intop_{0}^{\infty}u^{2\mu-\nu-\frac{1}{2}}\,K_{0}\left(u\cos(\delta)\right)\,du\nonumber \\
=\frac{c_{\nu-1}\,\Gamma(\nu)\,\Gamma^{2}\left(\mu-\frac{\nu}{2}+\frac{1}{4}\right)\,e^{-2\delta|\tau|}}{\sqrt{2}\,x^{\nu-\frac{1}{2}}(\cos(\delta))^{2\mu-\nu+\frac{1}{2}}},\,\,\,\,0\leq\delta<\frac{\pi}{2},\label{bound uniform}
\end{align}
where we have used the hypothesis that $\mu>\frac{\nu}{2}$,
as well as the formula for the Mellin transform of the function $K_{0}(x)$
{[}\cite{handbook_marichev}, p. 204, relation 3.14.1.3{]}. Thus, the series
on the left-hand side of (\ref{summation formula Olevskii}) converges absolutely and uniformly for
any $x\in(0,X]$.

Once again, we can now proceed with an application of the summation
formula (\ref{semyon formula Kin}). Indeed, 
\begin{align*}
\sum_{n\in\mathbb{Z}}|\Gamma(\mu+i\alpha n)|^{2}\,{}_{2}F_{1}\left(\mu+i\alpha n,\mu-i\alpha n;\nu;-x^{2}\right)\\
=2^{1-2\mu+\nu}x^{1-\nu}\Gamma(\nu)\,\intop_{0}^{\infty}u^{2\mu-\nu}J_{\nu-1}(xu)\,\sum_{n\in\mathbb{Z}}K_{2i\alpha n}(u)\,du,
\end{align*}
where the interchange of the orders of the infinite series and the
integral is justified by the uniform bound obtained in (\ref{bound uniform}).
Applying the summation formula (\ref{semyon formula Kin}), we find
that
\begin{align}
\sum_{n\in\mathbb{Z}}|\Gamma(\mu+i\alpha n)|^{2}\,{}_{2}F_{1}\left(\mu+i\alpha n,\mu-i\alpha n;\nu;-x^{2}\right)\nonumber\\
=\frac{2^{1-2\mu+\nu}\pi x^{1-\nu}\Gamma(\nu)}{\alpha}\,\intop_{0}^{\infty}u^{2\mu-\nu}J_{\nu-1}(xu)\,\sum_{n\in\mathbb{Z}}e^{-u\cosh\left(\frac{\pi n}{\alpha}\right)}\,du\nonumber\\
=\frac{2^{1-2\mu+\nu}\pi x^{1-\nu}\Gamma(\nu)}{\alpha}\,\sum_{n\in\mathbb{Z}}\intop_{0}^{\infty}u^{2\mu-\nu}J_{\nu-1}(xu)\,e^{-u\cosh\left(\frac{\pi n}{\alpha}\right)}du\nonumber\\
=\frac{\pi\,2^{2-2\mu}\,\Gamma\left(2\mu\right)}{\alpha}\,\sum_{n\in\mathbb{Z}}\frac{1}{\cosh^{2\mu}\left(\frac{\pi n}{\alpha}\right)}\,_{2}F_{1}\left(\mu,\mu+\frac{1}{2};\,\nu;\,-\frac{x^{2}}{\cosh^{2}\left(\frac{\pi n}{\alpha}\right)}\right), \label{calculations Olevskii}
\end{align}
where the interchange of the orders of operations 
can be justified by the following procedure
\begin{align*}
\sum_{n=1}^{\infty}\intop_{0}^{\infty}u^{2\mu-\nu}|J_{\nu-1}(xu)|\,e^{-u\cosh\left(\frac{\pi n}{\alpha}\right)}du & \leq\frac{1}{2\sqrt{2}}\,\sum_{n=1}^{\infty}\frac{1}{\sinh\left(\frac{\pi n}{2\alpha}\right)}\,\intop_{0}^{\infty}u^{2\mu-\nu-\frac{1}{2}}|J_{\nu-1}(xu)|\,e^{-u}\,du\\
 & \leq\frac{c_{\nu-1}}{2\sqrt{2x}}\,\sum_{n=1}^{\infty}\frac{1}{\sinh\left(\frac{\pi n}{2\alpha}\right)}\,\intop_{0}^{\infty}u^{2\mu-\nu-1}\,e^{-u}du<\infty,
\end{align*}
where we have adapted the first inequalities of (\ref{steps second justification template})
and used the fact that $|J_{\nu-1}(xu)|\leq c_{\nu-1}(xu)^{-\frac{1}{2}}$
and the hypothesis $\mu>\frac{\nu}{2}$.
Finally, in the last step of (\ref{calculations Olevskii}) we have invoked the integral representation
given in {[}\cite{handbook_marichev}, page 155, relation 3.10.3.1{]},
which is valid when $\mu>0$. This completes the proof of (\ref{summation formula Olevskii}).

\end{proof}

\begin{remark}
Taking different integral representations involving Gauss' hypergeometric
function with different arguments, one can show formulas analogous
to (\ref{summation formula Olevskii}) in many settings. Such formula
can be deduced for the infinite series
\begin{equation}
\sum_{n=1}^{\infty}\left|\Gamma\left(\mu+i\alpha n\right)\right|^{2}\,_{3}F_{2}\left(\mu+i\alpha n,\mu-i\alpha n,\nu;\rho,\eta;x\right),\label{form 3f2 sum}
\end{equation}
where we assume that $x,\mu,\alpha>0$. To prove such a formula, it
would be enough to invoke the Mellin-Barnes representation {[}\cite{MARICHEV}, Vol. III, 
p. 726, relation 8.4.50.2{]}
\begin{align}
\frac{\left|\Gamma\left(a+\frac{i\tau}{2}\right)\right|^{2}\Gamma(b)}{\Gamma(c)\Gamma(d)}x^{-a}\,_{3}F_{2}\left(a+\frac{i\tau}{2},a-\frac{i\tau}{2},b;c,d;-\frac{1}{x}\right)\nonumber \\
=\frac{1}{2\pi i}\,\intop_{\gamma-i\infty}^{\gamma+i\infty}\frac{\Gamma(a-s)\Gamma(b-a+s)}{\Gamma(c-a+s)\Gamma(d-a+s)}\,\Gamma\left(s+\frac{i\tau}{2}\right)\Gamma\left(s-\frac{i\tau}{2}\right)\,x^{-s}ds,\label{equation 3f2}
\end{align}
which is valid for $0<a<\gamma$. A summation formula of the form
(\ref{Summation formula product gammas}) would follow from (\ref{Summation formula product gammas})
and, again, from the Mellin-Barnes integral representation as presented
in {[}\cite{MARICHEV}, Vol. III, p. 726, relation 8.4.50.2{]}.
\end{remark}

\begin{remark}
Imposing certain restrictions on $\nu$, some interesting particular
cases of the previous formula might arise. For example, take $\nu=2\mu$:
then, for any $\mu>0$, the following summation formula holds
\begin{align*}
\sum_{n\in\mathbb{Z}}|\Gamma(\mu+i\alpha n)|^{2}\,&{}_{2}F_{1}\left(\mu+i\alpha n,\mu-i\alpha n;\,2\mu;-x^{2}\right)\\
=\frac{2\pi\Gamma\left(2\mu\right)}{\alpha}\sum_{n\in\mathbb{Z}}\frac{1}{\sqrt{x^{2}+\cosh^{2}\left(\frac{\pi n}{\alpha}\right)}}&\,\left(\frac{1}{\cosh\left(\frac{\pi n}{\alpha}\right)+\sqrt{x^{2}+\cosh^{2}\left(\frac{\pi n}{\alpha}\right)}}\right)^{2\mu-1},\label{Interesting formula as particular case}
\end{align*}
where we have used relation 7.3.1.104 on page 461 of the third volume of \cite{MARICHEV},
\begin{equation}
_{2}F_{1}\left(a,a+\frac{1}{2};2a;z\right)=\frac{1}{\sqrt{1-z}}\left(\frac{2}{1+\sqrt{1-z}}\right)^{2a-1}.\label{particular case 2f1}
\end{equation}
As another example, if $\nu=2\mu+1$ then, for any $\mu>0$, we have
the following formula
\begin{equation*}
\sum_{n\in\mathbb{Z}}|\Gamma(\mu+i\alpha n)|^{2}\,{}_{2}F_{1}\left(\mu+i\alpha n,\mu-i\alpha n;\,2\mu+1;-x^{2}\right)
=\frac{4\pi}{\alpha}\,\Gamma\left(2\mu\right)\sum_{n\in\mathbb{Z}}\left(\frac{\cosh\left(\frac{\pi n}{\alpha}\right)}{\cosh^{2}\left(\frac{\pi n}{\alpha}\right)+\sqrt{x^{2}+\cosh^{2}\left(\frac{\pi n}{\alpha}\right)}}\right)^{2\mu},
\end{equation*}
where the last expression is actually a consequence of the relation
7.3.1.105 on page 461 of [\cite{MARICHEV}, Vol. III], 
\begin{equation}
_{2}F_{1}\left(a,a+\frac{1}{2};2a+1;z\right)=\left(\frac{2}{1+\sqrt{1-z}}\right)^{2a}.\label{formula 2a+1}
\end{equation}

The reader is encouraged to find new examples of the summation formula
by looking directly at known reductions formulas for Gauss' hypergeometric
function $_{2}F_{1}(a,b;c;z)$. 

\end{remark}
\bigskip{}

\bigskip{}

In the next corollary we prove an interesting summation formula involving
the index of Lommel function $S_{\mu,i\tau}(x)$ (cf. section 11.8 of \cite{NIST}). Analogues of Schl\"omlich series involving Lommel functions have been considered in \cite{masirevic_lommel}. The connection between series of this type and the transformation formula for the Hurwitz zeta function was observed for the first time in \cite{dixit_kumar}. The reader can also check \cite{Yakubovich_Lommel}, where inversion formulas are presented for integral transforms associated with $S_{\mu,i\tau}(x)$. 

\begin{corollary} \label{corollary 6}
Let $\alpha,x>0$ and assume that $\mu<\frac{1}{2}$. Then the following summation
formula involving Lommel functions holds
\begin{align}
\Gamma^{2}\left(\frac{1-\mu}{2}\right)\,S_{\mu,0}(x)+2\sum_{n=1}^{\infty}\left|\Gamma\left(\frac{1-\mu+i\alpha n}{2}\right)\right|^{2}S_{\mu,i\alpha n}(x)\nonumber \\
=\frac{2^{\mu+1}\,\pi\,\Gamma\left(1-\mu\right)}{\alpha}\sqrt{x}\,\left\{ S_{\mu-\frac{1}{2},\frac{1}{2}}(x)+2\sum_{n=1}^{\infty}\cosh^{\frac{1}{2}}\left(\frac{2\pi n}{\alpha}\right)\,S_{\mu-\frac{1}{2},\frac{1}{2}}\left(x\cosh\left(\frac{2\pi n}{\alpha}\right)\right)\right\} .\label{Formula Lommel setting}
\end{align}

\end{corollary}

\begin{proof}
By {[}\cite{MARICHEV}, Vol. II, p. 346, relation (2.16.3.15){]}, the
Lommel function can be represented via the Widder transform,
\begin{equation}
\intop_{0}^{\infty}\frac{u^{-\mu}}{u^{2}+x^{2}}K_{\nu}(u)\,du=\left(2x\right)^{-\mu-1}\Gamma\left(\frac{1-\mu-\nu}{2}\right)\Gamma\left(\frac{1-\mu+\nu}{2}\right)S_{\mu,\nu}(x),\label{representation as starting point}
\end{equation}
where $\mu<1-|\text{Re}(\nu)|$. In particular, when $\mu<1$, one
has the following integral representation
\begin{equation}
\left|\Gamma\left(\frac{1-\mu+i\tau}{2}\right)\right|^{2}S_{\mu,i\tau}(x)=(2x)^{\mu+1}\intop_{0}^{\infty}\frac{u^{-\mu}}{u^{2}+x^{2}}K_{i\tau}(u)\,du.\label{Lommel widder imaginary index}
\end{equation}

If we take $\nu=\frac{1}{2}$ in (\ref{representation as starting point})
and combine the elementary relation $K_{\frac{1}{2}}(x)=\sqrt{\frac{\pi}{2x}}\,e^{-x}$
with Legendre's duplication formula, we see that, for any $\mu<1$,
\begin{align}
\intop_{0}^{\infty}\frac{u^{-\mu}}{u^{2}+x^{2}}e^{-u}\,du & =\sqrt{\frac{2}{\pi}}\,\left(2x\right)^{-\mu-\frac{1}{2}}\Gamma\left(\frac{1}{2}-\frac{\mu}{2}\right)\Gamma\left(1-\frac{\mu}{2}\right)S_{\mu-\frac{1}{2},\frac{1}{2}}(x)\nonumber \\
 & =x^{-\mu-\frac{1}{2}}\,\Gamma(1-\mu)S_{\mu-\frac{1}{2},\frac{1}{2}}(x).\label{Laplace Lommell setting}
\end{align}

From (\ref{Lommel widder imaginary index}), we can check immediately
that the series on the left of (\ref{Formula Lommel setting}) is
absolutely convergent, because, for any $\delta\in[0,\frac{\pi}{2})$ (recall (\ref{uniform estimate Macdonald index and X})),
\begin{align*}
\left|\Gamma\left(\frac{1-\mu+i\tau}{2}\right)\right|^{2}|S_{\mu,i\tau}(x)| & \leq(2x)^{\mu+1}e^{-\delta|\tau|}\,\intop_{0}^{\infty}\frac{u^{-\mu}}{u^{2}+x^{2}}\,K_{0}\left(u\cos(\delta)\right)\,du\\
 & \leq2^{\mu+1}x^{\mu-1}e^{-\delta|\tau|}\,\intop_{0}^{\infty}u^{-\mu}K_{0}\left(u\cos(\delta)\right)\,du\\
 & =\frac{x^{\mu-1}\Gamma^{2}\left(\frac{1-\mu}{2}\right)}{\cos^{1-\mu}(\delta)}e^{-\delta|\tau|}.
\end{align*}
Analogously, the infinite series on the right-hand side of (\ref{Formula Lommel setting})
also converges absolutely for any positive fixed $x$, due to the
bound for $S_{\mu,\nu}(x)$ (see (\ref{Decay property once more Lommel!}) above),
\[
|S_{\mu,\nu}(x)|=O\left(x^{\mu-1}\right),\,\,\,\,x\rightarrow\infty,
\]
which shows that 
\[
\left|\cosh^{\frac{1}{2}}\left(\frac{2\pi n}{\alpha}\right)\,S_{\mu-\frac{1}{2},\frac{1}{2}}\left(x\cosh\left(\frac{2\pi n}{\alpha}\right)\right)\right|=O\left(x^{\mu-\frac{3}{2}}\cosh^{\mu-1}\left(\frac{2\pi n}{\alpha}\right)\right),\,\,\,\,n\rightarrow\infty,
\]
so that the condition of our corollary $\mu<\frac{1}{2}$ assures the convergence in this case. \footnote{In fact, to assure the convergence it suffices to assume only that $\mu<1$.}

Using (\ref{Laplace Lommell setting}), we have that 
\begin{align*}
\sum_{n\in\mathbb{Z}}\left|\Gamma\left(\frac{1-\mu+i\alpha n}{2}\right)\right|^{2}S_{\mu,i\alpha n}(x) & =(2x)^{\mu+1}\sum_{n\in\mathbb{Z}}\intop_{0}^{\infty}\frac{u^{-\mu}}{u^{2}+x^{2}}K_{i\alpha n}(u)\,du\\
=(2x)^{\mu+1}\intop_{0}^{\infty}\frac{u^{-\mu}}{u^{2}+x^{2}}\sum_{n\in\mathbb{Z}}K_{i\alpha n}(u)\,du & =\frac{\pi}{\alpha}(2x)^{\mu+1}\intop_{0}^{\infty}\frac{u^{-\mu}}{u^{2}+x^{2}}\,\sum_{n\in\mathbb{Z}}e^{-u\cosh\left(\frac{2\pi n}{\alpha}\right)}\,du\\
=\frac{\pi(2x)^{\mu+1}}{\alpha}\,\sum_{n\in\mathbb{Z}}\intop_{0}^{\infty}\frac{u^{-\mu}}{u^{2}+x^{2}}e^{-u\cosh\left(\frac{2\pi n}{\alpha}\right)}du & =\frac{2^{\mu+1}\pi\Gamma(1-\mu)}{\alpha}\,\sqrt{x}\,\sum_{n\in\mathbb{Z}}\,\cosh^{\frac{1}{2}}\left(\frac{2\pi n}{\alpha}\right)\,S_{\mu-\frac{1}{2},\frac{1}{2}}\left(x\cosh\left(\frac{2\pi n}{\alpha}\right)\right),
\end{align*}
which proves our formula (\ref{Formula Lommel setting}). Note that, on the fourth equality, the interchange can be justified in the same way as in  Corollaries \ref{corollary 3} and \ref{corollary 5}, because
\[
\sum_{n=1}^{\infty}\intop_{0}^{\infty}\frac{u^{-\mu}}{u^{2}+x^{2}}e^{-u\cosh\left(\frac{2\pi n}{\alpha}\right)}\leq\frac{1}{2\sqrt{2}}\,\sum_{n=1}^{\infty}\frac{1}{\sinh\left(\frac{\pi n}{\alpha}\right)}\,\intop_{0}^{\infty}\frac{u^{-\mu-\frac{1}{2}}}{u^{2}+x^{2}}\,e^{-u}du<\infty,
\]
since $\mu<\frac{1}{2}$ by hypothesis.
\end{proof}

\bigskip{}
\bigskip{}

We now prove a result concerning an infinite series, with respect to the index, of a product of two Whittaker functions. 

\begin{corollary} \label{corollary 8}
For $x,\alpha>0$, the following formula takes place
\begin{align}
W_{-\frac{1}{4},0}(x)\,W_{\frac{1}{4},0}(x)+2\sum_{n=1}^{\infty}W_{-\frac{1}{4},i\alpha n}(x)\,W_{\frac{1}{4},i\alpha n}(x)\nonumber \\
=\frac{x}{\sqrt{2}\alpha}\left\{ e^{-\frac{x}{2}}K_{0}\left(\frac{x}{2}\right)+2\sum_{n=1}^{\infty}e^{-\frac{x}{2}\cosh\left(\frac{\pi n}{\alpha}\right)}\,K_{0}\left(\frac{x}{2}\cosh\left(\frac{\pi n}{\alpha}\right)\right)\right\} .\label{Summation formula product Whittaker different indices}
\end{align}
Moreover, we have the pair of Fourier transforms
\begin{equation}
\intop_{-\infty}^{\infty}W_{-\frac{1}{4},i\tau}(x)\,W_{\frac{1}{4},i\tau}(x)\,e^{-i\tau\xi}\,d\tau=\frac{x}{\sqrt{2}}e^{-\frac{x}{2}\cosh\left(\frac{\xi}{2}\right)}\,K_{0}\left(\frac{x}{2}\cosh\left(\frac{\xi}{2}\right)\right)\label{Direct Fourier transform Product Whittaker!}
\end{equation}
\begin{equation}
W_{-\frac{1}{4},i\tau}(x)\,W_{\frac{1}{4},i\tau}(x)=\frac{x}{2\sqrt{2}\pi}\intop_{-\infty}^{\infty}e^{-\frac{x}{2}\cosh\left(\frac{\xi}{2}\right)}\,K_{0}\left(\frac{x}{2}\cosh\left(\frac{\xi}{2}\right)\right)\,e^{i\tau\xi}\,d\xi\label{Inversion formula Product whittaker 1/4}
\end{equation}

\end{corollary}

\begin{proof}
We appeal to the integral representation {[}\cite{handbook_marichev}, p. 460,
relation 3.30.6.7{]},
\begin{equation}
W_{-\frac{1}{4},i\alpha n}(x)\,W_{\frac{1}{4},i\alpha n}(x)=\frac{1}{4\pi i}\,\intop_{\sigma-i\infty}^{\sigma+i\infty}\frac{\Gamma\left(s+1\right)\,\Gamma\left(\frac{s+1}{2}+i\alpha n\right)\Gamma\left(\frac{s+1}{2}-i\alpha n\right)}{\Gamma\left(\frac{s}{2}+\frac{3}{4}\right)\Gamma\left(\frac{s}{2}+\frac{5}{4}\right)}x^{-s}ds,\label{mellin barnes integral rep}
\end{equation}
where $\sigma>-1$. We take the summation of the previous expression
over $n\in\mathbb{Z}$. Using exactly the same justification as in the interchange performed of the orders of integration performed in (\ref{after the interchange}), it is not difficult to show that\footnote{In fact, (\ref{discrete interchange}) is nothing but a discrete analogue of (\ref{after the interchange}).}
\begin{equation}
\sum_{n\in\mathbb{Z}}W_{-\frac{1}{4},i\alpha n}(x)\,W_{\frac{1}{4},i\alpha n}(x)=\frac{1}{4\pi i}\,\intop_{\sigma-i\infty}^{\sigma+i\infty}\frac{\Gamma(s+1)}{\Gamma\left(\frac{s}{2}+\frac{3}{4}\right)\Gamma\left(\frac{s}{2}+\frac{5}{4}\right)}\,\sum_{n\in\mathbb{Z}}\Gamma\left(\frac{s+1}{2}+i\alpha n\right)\Gamma\left(\frac{s+1}{2}-i\alpha n\right)\,x^{-s}ds. \label{discrete interchange}
\end{equation}
Since $\text{Re}\left(s+1\right)=\sigma+1>0$ by hypothesis of the
Mellin-Barnes representation (\ref{mellin barnes integral rep}),
we can use (\ref{Summation formula product gammas}) to show that
\begin{align}
\sum_{n\in\mathbb{Z}}W_{-\frac{1}{4},i\alpha n}(x)\,W_{\frac{1}{4},i\alpha n}(x) & =\frac{1}{4\sqrt{\pi}i\alpha}\,\intop_{\sigma-i\infty}^{\sigma+i\infty}\frac{\Gamma(s+1)\Gamma\left(\frac{s+1}{2}\right)\Gamma\left(\frac{s}{2}+1\right)}{\Gamma\left(\frac{s}{2}+\frac{3}{4}\right)\Gamma\left(\frac{s}{2}+\frac{5}{4}\right)}\,\sum_{n\in\mathbb{Z}}\frac{1}{\cosh^{s+1}\left(\frac{\pi n}{\alpha}\right)}\,x^{-s}ds\nonumber \\
 & =\sqrt{\frac{\pi}{2}}\,\frac{1}{2\pi i\alpha}\,\intop_{\sigma-i\infty}^{\sigma+i\infty}\frac{\Gamma^{2}(s+1)}{\Gamma\left(s+\frac{3}{2}\right)}\,\sum_{n\in\mathbb{Z}}\frac{1}{\cosh^{s+1}\left(\frac{\pi n}{\alpha}\right)}\,x^{-s}ds,\label{Intermediate step in the prood}
\end{align}
where, at the last step, we have simply used Legendre's duplication
formula. It is clear that
\begin{align}
\intop_{\sigma-i\infty}^{\sigma+i\infty}\sum_{n\in\mathbb{Z}}\left|\frac{\Gamma(s+1)\Gamma\left(\frac{s+1}{2}\right)\Gamma\left(\frac{s}{2}+1\right)}{\Gamma\left(\frac{s}{2}+\frac{3}{4}\right)\Gamma\left(\frac{s}{2}+\frac{5}{4}\right)}\frac{1}{\cosh^{s+1}\left(\frac{\pi n}{\alpha}\right)}\,x^{-s}\right||ds|\nonumber \\
\leq\sum_{n\in\mathbb{Z}}\frac{x^{-\sigma}}{\cosh^{\sigma+1}\left(\frac{\pi n}{\alpha}\right)}\,\intop_{\sigma-i\infty}^{\sigma+i\infty}\left|\frac{\Gamma(s+1)\Gamma\left(\frac{s+1}{2}\right)\Gamma\left(\frac{s}{2}+1\right)}{\Gamma\left(\frac{s}{2}+\frac{3}{4}\right)\Gamma\left(\frac{s}{2}+\frac{5}{4}\right)}\right|\,|ds|<\infty,\label{interchange fubini in the proof whittaker +roduct}
\end{align}
where the absolute convergence of the series is justified by the fact
that $\sigma>-1$. Furthermore, when $|t|\gg1$, we know that
\[
\left|\frac{\Gamma\left(\sigma+1+it\right)\Gamma\left(\frac{\sigma+1+it}{2}\right)\Gamma\left(\frac{\sigma}{2}+1+\frac{it}{2}\right)}{\Gamma\left(\frac{\sigma}{2}+\frac{3}{4}+\frac{it}{2}\right)\Gamma\left(\frac{\sigma}{2}+\frac{5}{4}+\frac{it}{2}\right)}\right|\ll|t|^{\sigma}e^{-\frac{\pi|t|}{2}},
\]
fully justifying (\ref{interchange fubini in the proof whittaker +roduct}).
Thus, by Fubini's theorem, we obtain the summation formula 
\begin{align*}
\sum_{n\in\mathbb{Z}}W_{-\frac{1}{4},i\alpha n}(x)\,W_{\frac{1}{4},i\alpha n}(x) & =\sqrt{\frac{\pi}{2}}\,\sum_{n\in\mathbb{Z}}\frac{x}{2\pi i\alpha}\,\intop_{\sigma+1-i\infty}^{\sigma+1+i\infty}\frac{\Gamma^{2}(s)}{\Gamma(s+\frac{1}{2})}\,\left(x\cosh\left(\frac{\pi n}{\alpha}\right)\right)^{-s}ds\\
 & =\frac{x}{\alpha\sqrt{2}}\,\sum_{n\in\mathbb{Z}}e^{-\frac{x}{2}\cosh\left(\frac{\pi n}{\alpha}\right)}K_{0}\left(\frac{x}{2}\cosh\left(\frac{\pi n}{\alpha}\right)\right),
\end{align*}
where the last step consisted in an application of the Mellin-Barnes
integral {[}\cite{handbook_marichev}, page 209, relation 3.14.3.1{]},
\[
e^{-x}K_{\nu}(x)=\frac{1}{2\pi i}\intop_{\mu-i\infty}^{\mu+i\infty}\frac{\sqrt{\pi}}{2^{s}}\,\frac{\Gamma(s-\nu)\Gamma(s+\nu)}{\Gamma\left(s+\frac{1}{2}\right)}\,x^{-s}ds,\,\,\,\,\mu>|\text{Re}(\nu)|,
\]
with $\nu=0$. Finally, the proof of the formulas (\ref{Direct Fourier transform Product Whittaker!})
and (\ref{Inversion formula Product whittaker 1/4}) is an immediate
consequence of the Fourier inversion formula and steps which are analogous to those that led
to the proof of Lemma \ref{fourier transform Whittaker}. 
\end{proof}

Our paper is reaching its end. We conclude it by presenting a new
summation formula attached to products of Gamma functions and the
index of the Whittaker function. The kernel presented in this summation formula dates back to Wimp's work \cite{Wimp} and the reader can find some inversion formulas for integral and discrete transforms involving it in the references \cite{srivastava_yakubovich, Yakubovich_Whittaker_discrete}. 

\begin{corollary} \label{corollary 9}
Let $x,\alpha>0$ and assume that $\mu<\frac{1}{2}$. Then the following
summation formula takes place 
\begin{align}
\Gamma^{2}\left(\frac{1}{2}-\mu\right)\,W_{\mu,0}(x)+2\sum_{n=1}^{\infty}\left|\Gamma\left(\frac{1}{2}-\mu+i\alpha n\right)\right|^{2}\,W_{\mu,i\alpha n}(x)\nonumber \\
=\frac{\pi2^{\mu+\frac{1}{2}}\sqrt{x}}{\alpha}\,\Gamma(1-2\mu)\,\left\{ D_{2\mu-1}\left(\sqrt{2x}\right)+2\sum_{n=1}^{\infty}e^{\frac{x\sinh^{2}(\pi n/\alpha)}{2}}\,D_{2\mu-1}\left(\sqrt{2x}\,\cosh\left(\frac{\pi n}{\alpha}\right)\right)\right\} ,\label{Summation formula gammas whittaker}
\end{align}
where $D_{2\mu-1}(\cdot)$ denotes the parabolic cylinder function,
(\ref{cylinder function def}).
Moreover, we have the pair of Fourier transforms
\begin{equation}
\intop_{-\infty}^{\infty}\left|\Gamma\left(\frac{1}{2}-\mu+i\tau\right)\right|^{2}\,W_{\mu,i\tau}(x)\,e^{-i\tau\xi}\,d\tau=\pi\,2^{\mu+\frac{1}{2}}\sqrt{x}\,\Gamma(1-2\mu)\,e^{\frac{x\sinh^{2}(\xi/2)}{2}}\,D_{2\mu-1}\left(\sqrt{2x}\,\cosh\left(\frac{\xi}{2}\right)\right),\label{first pair product gammas}
\end{equation}
\begin{equation}
\left|\Gamma\left(\frac{1}{2}-\mu+i\tau\right)\right|^{2}\,W_{\mu,i\tau}(x)=2^{\mu-\frac{1}{2}}\,\sqrt{x}\,\Gamma(1-2\mu)\,\intop_{-\infty}^{\infty}\,e^{\frac{x\sinh^{2}(\xi/2)}{2}}\,D_{2\mu-1}\left(\sqrt{2x}\,\cosh\left(\frac{\xi}{2}\right)\right)e^{i\tau\xi}\,d\xi.\label{second pair product gammas}
\end{equation}

\end{corollary}

\begin{proof}
In contrast with the previous corollary, we compute the Fourier transform
(\ref{first pair product gammas}) and leave the details proof of
the summation formula to the reader.

We start with an integral representation that can be found in {[}\cite{MARICHEV}, Vol. II,
page 352, relation 2.16.8.4{]}
\begin{equation}
\left|\Gamma\left(\frac{1}{2}-\mu+i\tau\right)\right|^{2}W_{\mu,i\tau}(x)=2(4x)^{\mu}e^{-x/2}\intop_{0}^{\infty}t^{-2\mu}e^{-\frac{t^{2}}{4x}}\,K_{2i\tau}(t)\,dt,\label{formula starting point Gammas times Whittaker}
\end{equation}
which is valid for $x>0$, $\tau\in\mathbb{R}$ and $\mu<\frac{1}{2}$
(see also \cite{handbook_marichev}, p. 210, relation 3.14.3.10 and
\cite{Yakubovich_Whittaker_discrete}, pp. 106-107 for a proof of this
formula). Thus, the Fourier transform of the left-hand side of (\ref{formula starting point Gammas times Whittaker})
is
\begin{align}
G_{\mu,\xi}(x) & :=\intop_{-\infty}^{\infty}\left|\Gamma\left(\frac{1}{2}-\mu+i\tau\right)\right|^{2}\,W_{\mu,i\tau}(x)\,e^{-i\tau\xi}\,d\tau\nonumber \\
 & =2(4x)^{\mu}e^{-x/2}\intop_{-\infty}^{\infty}\intop_{0}^{\infty}t^{-2\mu}e^{-\frac{t^{2}}{4x}}\,K_{2i\tau}(t)\,dt\,e^{-i\tau\xi}\,d\tau.\label{First step in computation of the Fourier}
\end{align}
Since $\mu<\frac{1}{2}$ and $|K_{2i\tau}(t)|\leq e^{-2\delta|\tau|}\,K_{0}(t\,\cos(\delta))$
for any $0\leq\delta<\frac{\pi}{2}$, it is clear from the asymptotic
behaviour $K_{0}(x)=-\log(x)+O(1),$ $x\rightarrow0^{+}$ (cf. (\ref{asymptotic small equal to zero})), that the
double integral exists as an absolutely convergent one. In fact, fixing
some $\delta\in[0,\frac{\pi}{2})$, we have
\begin{align}
\intop_{-\infty}^{\infty}\intop_{0}^{\infty}\left|t^{-2\mu}e^{-\frac{t^{2}}{4x}}\,K_{2i\tau}(t)\right|\,dt\,d\tau & \leq\intop_{-\infty}^{\infty}\intop_{0}^{\infty}t^{-2\mu}e^{-\frac{t^{2}}{4x}}\,e^{-2\delta|\tau|}\,K_{0}(t\,\cos(\delta))\,dt\,d\tau\nonumber \\
=\frac{1}{\delta}\intop_{0}^{\infty}t^{-2\mu}e^{-\frac{t^{2}}{4x}}K_{0}\left(t\cos(\delta)\right)\,dt & \leq\frac{1}{\delta}\,\intop_{0}^{\infty}t^{-2\mu}K_{0}\left(t\cos(\delta)\right)\,dt\nonumber \\
=\delta^{-1}2^{-2\mu-1} & \cos^{2\mu-1}(\delta)\,\Gamma^{2}\left(\frac{1}{2}-\mu\right),\label{Justification set of inequalities}
\end{align}
which proves that we can reverse the orders of the integration in
(\ref{First step in computation of the Fourier}). Therefore, we are
able to reach the following Fourier transform
\begin{align*}
G_{\mu,\xi}(x) & =2(4x)^{\mu}e^{-x/2}\intop_{0}^{\infty}t^{-2\mu}e^{-\frac{t^{2}}{4x}}\intop_{-\infty}^{\infty}K_{2i\tau}(t)\,e^{-i\tau\xi}\,d\tau\,dt\\
 & =\pi(4x)^{\mu}e^{-x/2}\intop_{0}^{\infty}t^{-2\mu}e^{-\frac{t^{2}}{4x}}\,e^{-t\cosh\left(\frac{\xi}{2}\right)}\,dt\\
 & =\pi2^{\mu+\frac{1}{2}}\sqrt{x}\,\Gamma(1-2\mu)\,e^{\frac{x\sinh^{2}(\xi/2)}{2}}\,D_{2\mu-1}\left(\sqrt{2x}\,\cosh\left(\frac{\xi}{2}\right)\right),
\end{align*}
where we have invoked (\ref{first pair}) to get the integral with
respect to $\tau$ and the representation (\ref{integral cylinder useful final corollary})
of the cylinder function\footnote{which is valid because $\mu<\frac{1}{2}$ by hypothesis.}
to get the last evaluation. By the asymptotic behavior of the parabolic
cylinder function (\ref{Whittaker asymptotic formula}), we have
\begin{equation}
\left|e^{\frac{x\sinh^{2}(\xi/2)}{2}}\,D_{2\mu-1}\left(\sqrt{2x}\,\cosh\left(\frac{\xi}{2}\right)\right)\right|\sim(2x)^{\mu-\frac{1}{2}}e^{-\frac{x}{2}}\cosh^{2\mu-1}\left(\frac{\xi}{2}\right),\,\,\,\,\xi\rightarrow\infty,\label{asymptotic cylinder last result}
\end{equation}
and so, for a fixed $x\in(0,X]$, $g(\tau):=\left|\Gamma\left(\frac{1}{2}-\mu+i\tau\right)\right|^{2}W_{\mu,i\tau}(x)$
and $\hat{g}(\xi):=G_{\mu,\xi}(x)$ both belong to the set of functions
with moderate decrease, $\mathscr{F}$, because $\mu<\frac{1}{2}$
by hypothesis. Hence, the Fourier inversion formula is valid and we
have the pair of transforms (\ref{first pair product gammas}) and
(\ref{second pair product gammas}).

\bigskip{}

The summation formula results now from a straightforward application
of the Poisson summation formula or from the use of (\ref{semyon formula Kin})
in the integral representation (\ref{formula starting point Gammas times Whittaker}).
\end{proof}

\begin{remark}\label{remark product}

In the previous two corollaries we have respectively proved a formula involving
the product of two Whittaker functions and a formula involving
a combination of Gamma functions and the Whittaker function. The main
restriction in (\ref{Summation formula product Whittaker different indices}) is the fact that $\mu=\frac{1}{4},-\frac{1}{4}$. One may wonder if there exist summation
formulas (akin to the ones studied here) for the infinite series
\[
\sum_{n\in\mathbb{Z}}W_{\mu,in\alpha}(x)\,W_{\mu,in\alpha}(y),\,\,\,\,\,\,\sum_{n\in\mathbb{Z}}\left|\Gamma\left(\frac{1}{2}-\mu+in\alpha\right)\right|^{2}W_{\mu,in\alpha}(x)\,W_{\mu,in\alpha}(y),\,\,\,\,x,y>0,
\]
where we assume that $\mu<\frac{1}{2}$ in the latter infinite series.

One possible way to study these summation formulas is by invoking
the product formula \cite{sousa_guerra_yakubovich}\footnote{We should remark that this formula can be derived as a consequence
of (\ref{Fourier transform inverse Whittaker}), together with the injectivity of the Fourier transform and
the convolution theorem. This yields an interesting proof of it, which is much shorter than the proof presented in \cite{sousa_guerra_yakubovich}.}, 
\[
W_{\mu,i\tau}(x)\,W_{\mu,i\tau}(y)=2^{-\mu-1}\,\sqrt{\frac{xy}{\pi}}e^{\frac{x+y}{4}}\,\intop_{0}^{\infty}W_{\mu,i\tau}(u)\,u^{-\frac{3}{2}}\exp\left(\frac{u}{2}-\frac{(x+y)^{2}u}{8xy}-\frac{xy}{8u}\right)D_{2\mu}\left(\sqrt{\frac{xy}{2u}}+\frac{x+y}{\sqrt{2xy}}\sqrt{u}\right)\,du,
\]
or, alternatively, to use the equivalent representation (see relation
2.21.2.17 on page 321 of Vol. III of \cite{MARICHEV})
\[
W_{\mu,i\tau}(x)\,W_{\mu,i\tau}(y)=\frac{(xy)^{\mu}e^{-\frac{x+y}{2}}}{\Gamma(1-2\mu)}\intop_{0}^{\infty}e^{-w}w^{-2\mu}\,_{2}F_{1}\left(\frac{1}{2}-\mu-i\tau,\frac{1}{2}-\mu+i\tau;1-2\mu;\,-\frac{w\left(w+x+y\right)}{xy}\right)\,dw.
\]
We have tried to use the summation formulas (\ref{Theorem 1 summation formula Whittaker}) (resp. (\ref{summation formula Olevskii})) inside the previous representations for the product of Whittaker functions. However, these attempts of ours proved to be unsuccessful in this regard. 

\end{remark}

\bigskip{}

Our final corollary gives a variant of the summation formula (\ref{semyon formula Kin})
when we combine its general term $K_{i\alpha n}(x)$ with a rapidly
converging factor (it can be also regarded as an analogue of (\ref{corollary 4})). We omit its proof, because it only consists in taking $\mu=0$ in (\ref{Summation formula gammas whittaker}) and then use the reduction formula for Whittaker's function (\ref{D-1 formula!})
\begin{corollary}\label{corollary 7}
Let $x,\alpha>0$ and $\text{erfc}(\cdot)$ denote the complementary
error function (\ref{Definition complementary error function!}).
Then the following transformation formula holds
\begin{equation}
K_{0}(x)+2\sum_{n=1}^{\infty}\frac{K_{i\alpha n}(x)}{\cosh(\pi\alpha n)}=\frac{\pi}{\alpha}e^{x}\text{erfc}\left(\sqrt{2x}\right)+\frac{2\pi}{\alpha}\,\sum_{n=1}^{\infty}e^{x\cosh\left(\frac{2\pi n}{\alpha}\right)}\text{erfc}\left(\sqrt{2x}\cosh\left(\frac{\pi n}{\alpha}\right)\right).\label{particular case gamma product Whittaker alone}
\end{equation}

\end{corollary}


\bigskip{}
\bigskip{}
\bigskip{}

\textit{Disclosure Statement:} The authors declare that they have no conflict of interest

\bigskip{}

\textit{Funding:} This work was partially supported by CMUP, member of LASI, which is financed by national funds through FCT - Fundação para a Ciência e a Tecnologia, I.P., under the projects with reference UIDB/00144/2020 and UIDP/00144/2020. The first author also acknowledges the support from FCT (Portugal) through the PhD scholarship 2020.07359.BD. Since part of this work was made when the first author was a MSc student with a  scientific initiation scholarship with reference UIDB/00144/2020 (III), he acknowledges the support of CMUP through this scholarship.

\end{document}